\documentclass[10pt,reqno]{amsart}
\usepackage{amsmath,amssymb}
\usepackage{latexsym}
\usepackage{bm}
\usepackage{graphicx}

\newtheorem{thm}{Theorem}[section]

\newtheorem{lem}[thm]{Lemma}

\newtheorem{defn}{Definition}[section]

\newtheorem{rem}{Remark}[section]

\def\real{\rm Real}
\def\IR{\mathbb{R}}
\def\Span{\mathop{\rm Span}}
\def\Diag{\mathop{\rm Diag}}

\begin{document}

\title[Generalized explicit pseudo two-step Runge-Kutta-Nystr\"{o}m methods]{Generalized explicit pseudo two-step Runge-Kutta-Nystr\"{o}m methods for solving second-order initial value problems}

\author{N. S. Hoang}

\address{Department of Mathematics, University of West Georgia, Carrollton, GA 30118, USA}

\email{nhoang@westga.edu}

\date{}



\maketitle


\begin{abstract}
A class of explicit pseudo two-step Runge-Kutta-Nystr\"{o}m (GEPTRKN) methods for solving second-order initial value problems
$y'' = f(t,y,y')$, $y(t_0) = y_0$, $y'(t_0)=y'_0$  
has been studied. This new class of methods can be considered a generalized version of the class of classical explicit pseudo two-step Runge-Kutta-Nystr\"{o}m methods. 
We proved that an $s$-stage GEPTRKN method has step order of accuracy $p=s$ and stage order of accuracy $r=s$ for any set of distinct collocation parameters $(c_i)_{i=1}^s$. Super-convergence for order of accuracy of these methods can be obtained if the collocation parameters $(c_i)_{i=1}^s$ satisfy some orthogonality conditions. We proved that an $s$-stage GEPTRKN method can attain order of accuracy $p=s+2$. Numerical experiments have shown that the new methods work better than classical methods for solving non-stiff problems even on sequential computing environments. By their structures, the new methods will be much more efficient when implemented on parallel computers. 
\end{abstract}


\section{Introduction}

Consider the initial value problem
\begin{equation}
\label{sys}
y''(t) = f(t,y(t),y'(t)),\quad
y(t_0) = y_0,\quad y'(t_0) = y'_0,\qquad
t\in [t_0,\, t_0+T]
\end{equation}
where $f:[t_0,t_0+T]\times \IR\times \IR \to \IR$ and $f(t,y,z)$ is continuous with respect to $t$ and satisfies a 
Lipschitz condition with respect to $y$ and $z$.
For simplicity of notations we state equation \eqref{sys} in scalar form. However, the results in this paper remain valid when equation \eqref{sys} is in vector form. The common approach to solve numerically equation \eqref{sys} is to rewrite the 
equation as a system of first-order ODEs and use numerical methods to solve this system. Numerical methods for solving systems of first-order ODEs have been developed extensively in the literature (see, e.g., \cite{Burrage}, \cite{Butcher}, \cite{Hairer}, \cite{HW}). 
Among these methods, multistep methods and Runge-Kutta methods are the most frequently used. 
The drawback of this approach is that the sizes of the obtained systems are twice as large as the sizes of the original systems. 
Some numerical methods have been developed for solving equation \eqref{sys} directly \cite{VR}. However, to the author's knowledge no explicit {\it collocation} Runge-Kutta type method has been developed for solving equation \eqref{sys}. 

Methods that are designed to take advantage of a priori information from solutions to first-order and second-order initial value problems have also been studied considerably. These methods include exponential-fitted methods, trigonometrically-fitted methods, and functionally-fitted methods (see, e.g., \cite{Ozawa2}, \cite{Franco2}, \cite{Franco14A}, \cite{Franco14J}, \cite{simos}, \cite{nguyen1}, and \cite{nguyen3}). 

 A special form of equation \eqref{sys} that has received much of attention is the following one
\begin{equation}
\label{sys1}
y''(t) = f(t,y(t)),\quad
y(t_0) = y_0,\quad y'(t_0) = y'_0,\qquad
t\in [t_0,\, t_0+T].
\end{equation} 
Numerical methods for solving equation \eqref{sys1}, without rewriting it as a system of first-order ODEs, have been developed to a great extend in the literature (see, e.g., \cite{cong99}, \cite{cong01}, \cite{cong00}, \cite{cong91}, \cite{congminh07}). These methods are often referred to as direct methods for solving \eqref{sys1}. 
Among direct methods for solving \eqref{sys1}, Runge-Kutta-Nystr\"{o}m (RKN) methods are the most favorite one.  
An $s$-stage RKN method is defined by its Butcher-tableau as follows:
$$
\renewcommand{\arraystretch}{1.2}
\begin{array}{c|c}
\bm{c} & \bm{A} \\ \hline
        & \bm{b}^T\\
        & \bm{d}^T
\end{array},~~
\bm{A} = [a_{ij}] \in \IR^{s\times s},~~ 
\bm{b} = (b_1, ..., b_s)^T,~~ 
\bm{d} = (d_1, ..., d_s)^T,~~
\bm{c} = (c_1, ..., c_s)^T. 
$$ 
When $y_{n}$ and $y_n'$, the approximations of $y(t_n)$ and $y'(t_n)$ at the $n$-th step are available, 
the approximations of $y(t_{n+1})$ and $y'(t_{n+1})$ at the $(n+1)$-th step 
are defined by the $s$-stage RKN method with coefficients $(\bm{c},\bm{A},\bm{b},\bm{d})$ as follows
\begin{align}
y_{n+1} &= y_n + hy'_n +
h^2\sum_{j=1}^{s}b_jf(t_n+c_jh,Y_{n,j}), \label{RKNiterate}\\
y'_{n+1} &= y'_n + 
h\sum_{j=1}^{s}d_jf(t_n+c_jh,Y_{n,j}),\label{RKNiterate0}\\
Y_{n,i} &= y_n + c_ihy'_n+
h^2\sum_{j=1}^{s}a_{ij}f(t_n+c_jh,Y_{n,j}),\qquad i=1, ..., s\label{RKNstage}.
\end{align} 

If the matrix $\bm{A}$ is nonsingular then the method is called implicit as 
the stage values $(Y_{n,j})_{j=1}^s$ are defined implicitly in 
system \eqref{RKNstage}. This system is nonlinear and one has to solve for the stage values $(Y_{n,j})_{j=1}^s$ in numerical implementation by using  
Newton's method or fixed-point iterations.
When $\bm{A}$ is strictly 
lower-triangular and $c_1=0$, then corresponding method is called explicit as 
the stage values $(Y_{n,i})_{i=1}^s$ can be easily computed from the equations
\begin{equation*}
\label{explicitRKstage}
Y_{n,1} = y_n,\quad Y_{n,i} = y_n + hc_iy'_n +
h^2\sum_{j=1}^{i-1}a_{ij}f(t_n+c_jh,Y_{n,j}),\quad i=2, ..., s.
\end{equation*}
Once the stage values  $(Y_{n,j})_{j=1}^s$ are found, the numerical solutions $y_{n+1}$ and $y'_{n+1}$ are computed by \eqref{RKNiterate} and \eqref{RKNiterate0}.

\section{Explicit pseudo two-step RKN (EPTRKN) methods}
\label{sec:eptrk}

The iteration scheme \eqref{RKNiterate}--\eqref{RKNstage}
of a RKN method can be represented
in vector form as
\begin{subequations}
\begin{align}
y_{n+1} &= y_n + hy'_n+ h^2\bm{b}^Tf(\bm{e}t_n+\bm{c}h,\bm{Y}_n) \in \IR,\\
y'_{n+1} &= y'_n + h\bm{d}^Tf(\bm{e}t_n+\bm{c}h,\bm{Y}_n) \in \IR,\\
\label{eq6a}
\bm{Y}_n &= \bm{e}y_n+h\bm{c}y'_n+h^2\bm{A}f(\bm{e}t_n+\bm{c}h,\bm{Y}_n) \in \IR^s,
\end{align}
\end{subequations}
where 
$\bm{Y}_n:=(Y_{n,1},...,Y_{n,s})^T$ and 
$f(\bm{e}t_n+\bm{c}h,\bm{Y}_n) := 
(f(t_n+c_1h,Y_1), ..., f(t_n+c_sh,Y_s))^T$. As we mentioned before, implicit RKN methods require to solve nonlinear equation \eqref{eq6a} for the stage vector $\bm{Y}_n$ and this costs extra computational time. This is the case with classical collocation RKN methods as they are implicit \cite{cong91}. Due to the extra high computational cost of solving nonlinear systems for $\bm{Y}_n$, implicit Runge-Kutta and Runge-Kutta-Nystr\"{o}m methods should only be used for solving stiff problems. 
For non-stiff problems, explicit methods are computationally cheaper as the stage values $Y_{n,i}$ can be consequentially computed without solving any equation.

In \cite{cong99} a class of explicit pseudo two-step RKN methods was studied. 
The iteration scheme of an $s$-stage explicit pseudo two-step RKN (EPTRKN) method was defined as 
\begin{subequations}
\label{EPTRKN}
\begin{align}
y_{n+1} &= y_{n} + hy'_n+ h^2\bm{b}^Tf(\bm{e}t_{n}+\bm{c}h,\bm{Y}_{n}) \in \IR,\\
y'_{n+1} &= y'_{n} + h\bm{d}^Tf(\bm{e}t_{n}+\bm{c}h,\bm{Y}_{n}) \in \IR,\\
\label{EPTRKNc}
\bm{Y}_{n+1} &= \bm{e}y_{n+1} + h\bm{c}y'_{n+1}+h^2\bm{A}f(\bm{e}t_{n}+\bm{c}h,\bm{Y}_{n}) \in \IR^s,
\end{align}
\end{subequations}
where  
$y_n\approx y(t_n)$, $y'_n\approx y'(t_n)$, and 
$\bm{Y}_n=(Y_{n,1},...,Y_{n,s})^T\approx y(\bm{e}t_n+\bm{c}h)=(y(t_n+c_1h),...,y(t_n+c_sh))^T$. 
The main advantage of EPTRKN methods over implicit RKN methods is that they are explicit. Specifically, the stage vector $\bm{Y}_{n+1}$ in equation \eqref{EPTRKNc} is computed explicitly using the values of $y_n$, $y_n'$, and $\bm{Y}_{n}$ from the previous step. 
To start the scheme one needs $s$ sufficiently accurate starting values 
to define $\bm{Y}_0$ and these values can be obtained by any conventional method. 
By construction EPTRKN methods are ideally suited 
for parallel computers as the components of $f(\bm{e}t_{n}+\bm{c}h,\bm{Y}_{n})$ can be
evaluated independently in parallel computing environments. Consequently, in parallel computing environments, EPTRKN methods use only one function evaluation of $f(t_{n}+c_ih,Y_{n,i})$ per step. 

The advantage of using EPTRKN methods for solving non-stiff second-order ODEs in the special form \eqref{sys1} has been demonstrated in the literature \cite{cong99}, \cite{cong00}, \cite{cong01}. However, these methods are not applicable to non-stiff initial-value problems in the general form \eqref{sys}. Thus, our goal in this paper is to develop a new  class of methods based on EPTRKN methods for solving equation \eqref{sys}. The new methods will be called generalized EPTRKN methods (GEPTRKN) to differentiate them from the regular EPTRKN methods.

\section{Collocation generalized explicit pseudo two-step RKN (GEPTRKN) methods}
\label{sec:EPTRKN}

\subsection{Generalized EPTRKN methods}

Given the values $y_n\approx y(t_n)$, $y'_n\approx y'(t_n)$, $\bm{Y}_n\approx y(\bm{e}t_n + \bm{c}h)$, and $\bm{Y}'_n\approx y'(\bm{e}t_n + \bm{c}h)$ at the time step $t_{n}$, the approximate values $y_{n+1}$, $y'_{n+1}$, $\bm{Y}_{n+1}$, and $\bm{Y}'_{n+1}$ at the time step $t_{n+1}=t_n +h$ are computed by an $s$-stage generalized explicit pseudo two-step RKN (GEPTRKN) method with coefficients $(\bm{c}, \bm{A}, \bm{B}, \bm{b}, \bm{d})$ as follows
\begin{align*}
y_{n+1} &= y_{n} + hy'_n+ h^2\bm{b}^Tf(\bm{e}t_{n}+\bm{c}h,\bm{Y}_{n},\bm{Y}'_{n}) \in \IR,\\
y'_{n+1} &= y'_{n} + h\bm{d}^Tf(\bm{e}t_{n}+\bm{c}h,\bm{Y}_{n},\bm{Y}'_{n}) \in \IR,\\
\label{GEPTRKNc}
\bm{Y}_{n+1} &= \bm{e}y_{n+1} + h\bm{c}y'_{n+1}+h^2\bm{A}f(\bm{e}t_{n}+\bm{c}h,\bm{Y}_{n},\bm{Y}'_{n}) \in \IR^s,\\
\bm{Y}'_{n+1} &= \bm{e}y'_{n+1}+h\bm{B}f(\bm{e}t_{n}+\bm{c}h,\bm{Y}_{n},\bm{Y}'_{n}) \in \IR^s.
\end{align*}
Here $\bm{A}$ and $\bm{B}$ are square matrices of size $s\times s$ and $\bm{b}$ and $\bm{d}$ are vectors in $\mathbb{R}^s$. 
The parameters $(c_i)_{i = 1}^s$ are distinct and will be chosen later. We assume that the initial stage vectors $\bm{Y}_0$ and $\bm{Y}'_0$ 
are available at high accuracy. This can be obtained by using classical Runge-Kutta methods to solve for $\bm{Y}_0$ and $\bm{Y}'_0$. 
By construction, GEPTRKN methods share the same structure with EPTRKN methods and they are explicit.

To determine coefficients of GEPTRKN methods we first define the following operators:
\begin{equation*}
\begin{split}
\mathcal{L}(u)(t) &:= u(t+h) - u(t) - hu'(t) - h^2\bm{b}^Tu''(\bm{e}t+\bm{c}h),\\
\mathcal{M}(u)(t) &:= u'(t+h) - u'(t) - h\bm{d}^Tu''(\bm{e}t+\bm{c}h),\\
\mathcal{N}(u)(t) &:= u(\bm{e}t+\bm{c}h+\bm{e}h) - \bm{e}u(t+h) - h\bm{c}u'(t+h) - h^2\bm{A}u''(\bm{e}t+\bm{c}h),\\
\mathcal{O}(u)(t) &:= u'(\bm{e}t+\bm{c}h+\bm{e}h) - \bm{e}u'(t+h) - h\bm{B}u''(\bm{e}t+\bm{c}h).
\end{split}
\end{equation*}

\begin{defn}[Collocation GEPTRKN]
\label{def3.1}
An s-stage GEPTRKN method with coefficients $(\bm{c}, \bm{A}, \bm{B}, \bm{b}, \bm{d})$ is called a collocation GEPTRKN method if the following equation holds
\begin{equation}
\label{GEPTRKNdef}
\begin{split}
\mathcal{L}(t^{k+2}) = \mathcal{M}(t^{k+2}) = 0, \qquad \mathcal{O}(t^{k+2}) = \mathcal{N}(t^{k+2}) = \vec{\bm{0}},\qquad k=0,1,...,s-1.
\end{split}
\end{equation}
\end{defn}

From now on by GEPTRKN methods we mean {\it collocation} GEPTRKN methods.
Given the parameters $(c_i)_{i=1}^s$, the coefficients $(\bm{c},\bm{A}, \bm{B}, \bm{b}, \bm{d})$ 
of an $s$-stage GEPTRKN method can be found from the equations in \eqref{GEPTRKNdef}. 
Specifically, from the equations in \eqref{GEPTRKNdef} one obtains
\begin{equation}
\label{sysABeq}
\begin{split}
1
&= (k+2)(k+1)\bm{b}^T\bm{c}^k, \qquad k = 0,...,s-1,\\
1
&= (k+1)\bm{d}^T\bm{c}^k, \qquad k = 0,...,s-1,\\
\bm{c}^{k+2}
&= (k+2)(k+1)\bm{A}(\bm{c} - \bm{e})^k, \qquad k = 0,...,s-1,\\
\bm{c}^{k+1}
&= (k+1)\bm{B}(\bm{c} - \bm{e})^k, \qquad k = 0,...,s-1.
\end{split}
\end{equation}
System \eqref{sysABeq} can be rewritten in vector form as
\begin{equation}
\label{feq12}
\begin{split}
\begin{bmatrix}
\frac{1}{2} & \frac{1}{3\times 2} &\cdots &\frac{1}{(s+1)\times s}
\end{bmatrix}
&= \bm{b}^T\begin{bmatrix}\bm{e} &\bm{c} &\bm{c}^2 &\cdots &\bm{c}^{s-1}\end{bmatrix},\\
\begin{bmatrix}
\frac{1}{1} & \frac{1}{2} &\cdots &\frac{1}{s}
\end{bmatrix}
&= \bm{d}^T\begin{bmatrix}\bm{e} &\bm{c} &\bm{c}^2 &\cdots &\bm{c}^{s-1}\end{bmatrix},\\
\begin{bmatrix}
\frac{\bm{c}^2}{2} & \frac{\bm{c}^3}{3\times2} &\cdots &\frac{\bm{c}^{s+1}}{(s+1)\times s} 
\end{bmatrix}
&= \bm{A}\begin{bmatrix}\bm{e} &(\bm{c}-\bm{e}) &(\bm{c}-\bm{e})^2 &\cdots &(\bm{c} - \bm{e})^{s-1}\end{bmatrix},\\
\begin{bmatrix}
\frac{\bm{c}}{1} & \frac{\bm{c}^2}{2} &\cdots& \frac{\bm{c}^{s}}{s} 
\end{bmatrix}
&= \bm{B}\begin{bmatrix}\bm{e} &(\bm{c}-\bm{e}) &(\bm{c}-\bm{e})^2 &\cdots &(\bm{c} - \bm{e})^{s-1}\end{bmatrix}.
\end{split}
\end{equation}
Note that
$$
\begin{bmatrix}\bm{e} &\bm{c} &\bm{c}^2 &\cdots &\bm{c}^{s-1}\end{bmatrix} = 
\begin{bmatrix}
1  &c_1  & \cdots &c_1^{s-1}\\
1  &c_2  & \cdots &c_2^{s-1}\\
\vdots    & \vdots    & \ddots & \vdots\\
1  &c_s  & \cdots &c_s^{s-1}
\end{bmatrix}
$$
which is a Vandermonde matrix and, therefore, 
it is invertible if $(c_i)_{i=1}^s$ are distinct. By the same reason, the matrix
$$
\begin{bmatrix}\bm{e} &(\bm{c}-\bm{e}) &(\bm{c}-\bm{e})^2 &\cdots &(\bm{c}-\bm{e})^{s-1}\end{bmatrix} = 
\begin{bmatrix}
1  &(c_1-1)  & \cdots &(c_1-1)^{s-1}\\
1  &(c_2-1)  & \cdots &(c_2-1)^{s-1}\\
\vdots    & \vdots    & \ddots & \vdots\\
1  &(c_s-1)  & \cdots &(c_s-1)^{s-1}
\end{bmatrix}
$$
is also invertible if $(c_i)_{i=1}^s$ are distinct. Thus, the coefficients $(\bm{c}, \bm{A}, \bm{B}, \bm{b}, \bm{d})$ can be found uniquely by solving the linear systems in \eqref{feq12}. 

\subsection{The collocation solution}

Let $(c_i)_{i=1}^s$ be distinct values and let
$$\mathcal{P}_{s+1}:= \Span\{1,t,t^2,...,t^{s+1}\} :=
\bigg\{ \sum_{i=0}^{s+1}a_it^{i}\,\bigg{|} \, a_i \in \mathbb{R}, \, i=0,...,s+1\bigg\}.
$$
Given the values $y_n$, $y_n'$, $(Y_{n,i})_{i=1}^s$, and $(Y'_{n,i})_{i=1}^s$, we call $u(t)$ the {\em collocation solution}
if $u(t)\in \mathcal{P}_{s+1}$ and $u(t)$ satisfies the following equations
\begin{equation}
\label{interpol}
\begin{split}
u(t_n) &= y_n,\\
u'(t_n) &= y'_n,\\
u''(t_n+c_ih) &= f(t_n+c_ih,Y_{n,i},Y'_{n,i}),\qquad i = 1, ..., s.
\end{split}
\end{equation}
If such a polynomial $u(t)$ exists, then the values $y_{n+1}$, $y'_{n+1}$, $(Y_{n+1,i})_{i=1}^s$, and $(Y'_{n+1,i})_{i=1}^s$ 
at the $(n+1)$-th step are defined by
\begin{equation}
\label{iteratepoly}
\begin{split}
y_{n+1} &:= u(t_{n+1}),\quad y'_{n+1} := u'(t_{n+1}),\qquad t_{n+1}=t_n+h,\\
Y_{n+1,i} &:= u(t_{n+1}+c_ih),\quad Y'_{n+1,i} := u'(t_{n+1}+c_ih), \qquad i=1,...,s.
\end{split}
\end{equation}
Equations \eqref{interpol} and \eqref{iteratepoly} are called a {\it collocation method} for 
integrating equation \eqref{sys}. 
When $y_n$, $y'_n$, $\bm{Y}_n= (Y_{n,1}, ..., Y_{n,s})^T$, and $\bm{Y}'_n= (Y'_{n,1}, ..., Y'_{n,s})^T$  are available, 
$u(t)$ can be constructed explicitly  
through an interpolation polynomial involving $y_n$, $y'_n$, and 
$f(\bm{e}t_n+\bm{c}h,\bm{Y}_n,\bm{Y}'_n)$.
The existence and uniqueness of the collocation solution $u(t)$ satisfying \eqref{interpol} 
is justified in the following result.

\begin{lem}\label{newinter}
Suppose that the $s+2$ values $z_0,z'_0,z_1,z_2, ..., z_s$ are given 
and $(c_i)_{i=1}^s$ are distinct,  
then there exists an interpolation polynomial $\varphi \in\mathcal{P}_{s+1}$ such 
that 
\begin{equation}
\label{ae11}
\varphi(0) = z_0,\quad \varphi'(0) = z'_0,\quad \varphi''(c_ih)=z_i, \qquad i=1,...,s.
\end{equation}

\end{lem}
\begin{proof}
Since $\varphi \in \mathcal{P}_{s+1}$, it has the following form
$$
\varphi(t)=\sum_{i=0}^{s+1}a_it^i.
$$
For this representation of $\varphi(t)$, equation \eqref{ae11}
can be written as
\begin{equation}
\label{mq1}
\begin{bmatrix}
1&0&0 & 0 & \cdots & 0\\
0&1&0 & 0  & \cdots & 0\\
0&0&2(c_1h)^0  & 6(c_1h)^1  & \cdots & (s+1)s(c_1h)^{s-1}\\
\vdots&\vdots&\vdots    & \vdots    & \ddots & \vdots\\
0&0&2(c_sh)^0  & 6(c_sh)^1  & \cdots & (s+1)s(c_sh)^{s-1}
\end{bmatrix}
\begin{bmatrix}
a_0\\
a_1\\
a_2\\
\vdots\\
a_{s+1}
\end{bmatrix}
=
\begin{bmatrix}
z_0\\
z'_0\\
z_1\\
\vdots\\
z_s
\end{bmatrix}.
\end{equation}
Using cofactor expansions and the linearity of determinants with respect to columns one can show that the determinant of the left-hand side matrix in equation \eqref{mq1} is
\begin{equation}
\label{veq36}
\begin{vmatrix}
2  &6c_1h  & \cdots &(s+1)s(c_1h)^{s-1}\\
2  &6c_2h  & \cdots &(s+1)s(c_2h)^{s-1}\\
\vdots    & \vdots    & \ddots & \vdots\\
2  &6c_sh  & \cdots &(s+1)s(c_sh)^{s-1}
\end{vmatrix} 
= s!(s+1)! h^{\frac{s(s-1)}{2}}
\begin{vmatrix}
1  &c_1  & \cdots &c_1^{s-1}\\
1  &c_2  & \cdots &c_2^{s-1}\\
\vdots    & \vdots    & \ddots & \vdots\\
1  &c_s  & \cdots &c_s^{s-1}
\end{vmatrix}.
\end{equation}
The matrix on the right-hand side of equation \eqref{veq36} is a Vandermonde matrix which is known non-singular if $(c_i)_{i=1}^s$ are distinct. Thus, the determinant of the left-hand side matrix in equation \eqref{mq1} is non-singular. This implies that equation \eqref{mq1} has a unique solution. 
Therefore, the polynomial $\varphi(t)\in \mathcal{P}_{s+1}$ satisfying \eqref{ae11} exists and is unique. 
\end{proof}

\begin{thm}
\label{collofuns}
The collocation method \eqref{interpol}--\eqref{iteratepoly} 
is equivalent to the $s$-stage GEPTRKN method with coefficients 
$(\bm{c}$,$\bm{A}$,$\bm{B}$,$\bm{b}$,$\bm{d})$ defined by \eqref{feq12}.
\end{thm}

\begin{proof}
By Lemma~\ref{newinter} there exists a unique interpolation
polynomial $u(t)\in \mathcal{P}_{s+1}$ such that 
\begin{equation}
\label{mr1}
u(t_{n})=y_{n},\quad u'(t_{n})=y'_{n},\quad 
 u''(t_{n}+c_ih)=f(t_{n}+ c_ih,Y_{n,i},Y'_{n,i}),\qquad i=1,...,s. 
\end{equation}
This polynomial $u(t)$ is the collocation solution satisfying equation \eqref{interpol}. Let us verify that if we use $u(t)$ to generate the quantities in \eqref{iteratepoly},
then the following equations hold
\begin{equation}
\label{eq10}
\begin{split}
y_{n+1} &= y_{n} +hy'_n+ h^2\bm{b}^Tf(\bm{e}t_{n}+\bm{c}h,\bm{Y}_{n},\bm{Y}'_{n}),\\
y'_{n+1} &= y'_n+ h\bm{d}^Tf(\bm{e}t_{n}+\bm{c}h,\bm{Y}_{n},\bm{Y}'_{n}),\\
\bm{Y}_{n+1} &= \bm{e}y_{n+1} +h\bm{c}y'_{n+1} + h^2\bm{A}f(\bm{e}t_{n}+\bm{c}h,\bm{Y}_{n},\bm{Y}'_{n}),\\
\bm{Y}'_{n+1} &= \bm{e}y'_{n+1} + h\bm{B}f(\bm{e}t_{n}+\bm{c}h,\bm{Y}_{n},\bm{Y}'_{n}).
\end{split}
\end{equation}
Recall that the equations in \eqref{eq10} are used to define a GEPTRKN method with coefficients $(\bm{c}$,$\bm{A}$,$\bm{B}$,$\bm{b}$,$\bm{d})$. 
Since $u(t) \in \mathcal{P}_{s+1}$, it can be represented as
\begin{equation*}
\label{phicomb}
u(t)= \sum_{i=0}^{s+1}a_it^i.
\end{equation*}
By Definition \ref{def3.1},
the coefficients $(\bm{c},  \bm{A}, \bm{B}, \bm{b},\bm{d})$ of a GEPTRKN method  
satisfy the following equations
\begin{subequations}
\begin{align}
\label{eq12a}
\varphi(t_{n}+h) &= \varphi(t_{n}) + h \varphi'(t_{n})+
h^2\bm{b}^T\varphi''(\bm{e}t_{n}+\bm{c}h),\\
\varphi'(t_{n}+h) &= \varphi'(t_{n}) +
h\bm{d}^T\varphi''(\bm{e}t_{n}+\bm{c}h),\\
\varphi(\bm{e}t_n+\bm{e}h+\bm{c}h) &= \bm{e}\varphi(t_n+h) + \bm{c}h\varphi'(t_n+h)+
h^2\bm{A}\varphi''(\bm{e}t_n+\bm{c}h),\\
\varphi'(\bm{e}t_n+\bm{e}h+\bm{c}h) &= \bm{e}\varphi'(t_n+h)+
h\bm{B}\varphi''(\bm{e}t_n+\bm{c}h),
\end{align}
\end{subequations}
for $\varphi(t) = t^i$, $i = 0, ..., s+1$.
It follows from equation \eqref{eq12a} and the fact that $u(t)$ is a linear 
combination of $(t^k)_{k=0}^{s+1}$ that
\begin{align}
\label{cuoi2}
u(t_{n} + h)=u(t_{n})+hu'(t_{n})+ h^2\bm{b}^Tu''(\bm{e}t_{n}+\bm{c}h).
\end{align}
From equation \eqref{iteratepoly}, equation \eqref{mr1}, and equation
\eqref{cuoi2} one gets
\begin{align*}
\label{equivalent1}
y_{n+1} = u(t_{n}+h) = y_n + hy'_n +
h^2\bm{b}^Tf(\bm{e}t_{n}+\bm{c}h,\bm{Y}_{n},\bm{Y}'_{n}).
\end{align*}
Therefore, the first equation in \eqref{eq10} holds. The other equations in \eqref{eq10} can be obtained similarly. This completes the proof of Theorem \ref{collofuns}. 
\end{proof}

\section{Accuracy and stability properties}
\label{sec:accuracy}

\subsection{Order of accuracy}

Let us first introduce a definition for the stage order of accuracy and step order of accuracy of GEPTRKN methods.

\begin{defn}
\label{def4.1}
A GEPTRKN method is said to have step order $p=\min\{p_1,p_2\}$ 
and stage order $r=\min\{p_1,p_2,p_3,p_4\}$ if
\begin{align*}
y(t_{n+1})-y_{n+1}&=O(h^{p_1+1}),\\
y'(t_{n+1})-y'_{n+1}&=O(h^{p_2+1}),\\
y(\bm{e}t_{n+1}+\bm{c}h)-\bm{Y}_{n+1}&=O(h^{p_3+1}),\\
y'(\bm{e}t_{n+1}+\bm{c}h)-\bm{Y}'_{n+1}&=O(h^{p_4+1}),
\end{align*}
given that
$y_{n} = y(t_{n})+O(h^{p_1+1})$, $y'_{n} = y'(t_{n})+O(h^{p_2+1})$, $y(\bm{e}t_{n}+\bm{c}h)-\bm{Y}_{n}=O(h^{p_3+1})$, and $y'(\bm{e}t_{n}+\bm{c}h)-\bm{Y}'_{n}=O(h^{p_4+1})$.
\end{defn}

\begin{rem}\label{smooth1}\rm 
In \cite[Remark 3.1]{nguyen2} it was showed that
if a function $f(t)\in C^{m+n}[t_0,\,t_0+T]$ satisfies  
$f(t_i)=0,i=1,...,n$, then there exists $g(t)\in C^{m}[t_0,\,t_0+T]$ such that 
$f(t)=g(t)\prod_{i=1}^n(t-t_i)$. This result will be used for our study of order of accuracy of GEPTRKN methods below. 
\end{rem}

\begin{thm}\label{genorder}
An s-stage GEPTRKN method has stage order $r=s$ and 
step order $p=s$ for any set of collocation parameters $(c_i)_{i=1}^s$.
\end{thm}

\begin{proof}
Without loss of generality we assume that $t_{n}=0$. 
In addition, we can assume that $y_n = y(t_n)$, $y'_n = y'(t_n)$, $y(\bm{c}h)-\bm{Y}_{n} = O(h^{s+1})$ and 
$y'(\bm{c}h)-\bm{Y}'_{n} = O(h^{s+1})$.
Let $u(t)\in \mathcal{P}_{s+1}$ be the collocation solution satisfying equation \eqref{interpol}, i.e., 
\begin{equation}
\label{ae9}
u(t_{n})=y_{n},\quad u'(t_{n})=y'_{n},\quad
u''(\bm{c}h)=f(\bm{c}h,\bm{Y}_{n},\bm{Y}'_n).
\end{equation}
This collocation solution exists by Lemma~\ref{newinter}  
and from equation \eqref{iteratepoly} we have
$$
y_{n+1} = u(t_{n+1}),\quad y'_{n+1}=u'(t_{n+1}),\quad \bm{Y}_{n+1}=u(\bm{e}h+\bm{c}h),\quad 
\bm{Y}'_{n+1}=u'(\bm{e}h+\bm{c}h).
$$
It follows from the equation $y''(t) = f(t,y(t),y'(t))$, $\forall t\in [t_0,t_0 + T]$, the third equation in \eqref{ae9}, the local assumption 
$\bm{Y}_{n}-y(\bm{c}h) = O(h^{s+1})$, $\bm{Y}'_{n}-y'(\bm{c}h) = O(h^{s+1})$,
and Taylor expansions of $f(t,y,z)$ with respect to $y$ and $z$ that
\begin{equation}
\label{ae10}
\begin{split}
u''(\bm{c}h)-y''(\bm{c}h)&=
f(\bm{c}h,\bm{Y}_{n},\bm{Y}'_n)-f(\bm{c}h,y(\bm{c}h),y'(\bm{c}h))\\
&=
C(\bm{Y}_{n}  - y(\bm{c}h)) + D(\bm{Y}'_n-y'(\bm{c}h)) + O(h^{s+1})
= O(h^{s+1})
\end{split}
\end{equation}
where $C$ and $D$ are constants independent of $h$. 
Define 
\begin{equation}
\label{feq18}
R(t):=u(t)-y(t),\qquad 0\le t\le t_0 + T. 
\end{equation}
It follows from equations \eqref{ae9}, \eqref{ae10}, and \eqref{feq18} that
$$
R(0)=0,\qquad R'(0)=0,\qquad R''(c_ih)=u''(c_ih) - y''(c_ih) =O(h^{s+1}),\qquad i=1,...,s.
$$
Since $R''(t)$ is sufficiently smooth and $R''(c_ih)=O(h^{s+1})$, $i=1,...,s$, there exists a sufficiently smooth function $w(t)$
such that 
\begin{equation}
\label{m20a}
R''(t)=w(t)+O(h^{s+1}),\qquad  w(c_ih)=0,\qquad i=1,..,s.
\end{equation}
Since $w(c_ih)=0$, $i=1,...,s$, there exists a sufficiently smooth function $g(t)$, as we mentioned in Remark~\ref{smooth1}, such that
$$
w(t) = g(t)\prod_{i=1}^s(t-c_ih).
$$
This and equation \eqref{m20a} imply
\begin{equation}
\label{eq14}
R''(t) = g(t)\prod_{i=1}^s(t-c_ih) + O(h^{s+1}).
\end{equation}
Integrate equation \eqref{eq14} from 0 to $x$ and use the relation $R'(0)=0$ to obtain
\begin{equation}
\label{eqf15}
R'(x) = \int_0^x R''(t)\, dt = \int_0^x \bigg(g(t)\prod_{i=1}^s(t-c_ih) + O(h^{s+1})\bigg)\, dt.
\end{equation}
Letting $x = h$ in equation \eqref{eqf15} and using the substitution $t=\xi h$ one gets
\begin{equation}
\label{eq16}
R'(h) = h^{s+1}\int_0^1 g(\xi h)\prod_{i=1}^s(\xi-c_i)\, d\xi + 
\int_0^{h} O(h^{s+1})\, dt. 
\end{equation}
From equation \eqref{eq16} and the relation $\int_0^{h} O(h^{s+1})\, dt = O(h^{s+2})$, one obtains
\begin{equation}
\label{ae6}
y'_{n+1}-y'(t_{n+1})
=u'(t_{n+1})-y'(t_{n+1})=R'(h) = O(h^{s+1}).
\end{equation}
Let $x = \bm{e}h + \bm{c}h$ in equation \eqref{eqf15} and then substitute $t = \xi h$ to get
\begin{equation*}
\begin{split}
R'(\bm{e}h + \bm{c}h) &= \int_{\bm{0}}^{\bm{e}h + \bm{c}h}\bigg(g(t)\prod_{i=1}^s(t-c_ih) + O(h^{s+1})\bigg)\, dt\\
&= h^{s+1}\int_{\bm{0}}^{\bm{e} + \bm{c}}g(\xi h)\prod_{i=1}^s(\xi-c_i)\, d\xi + \int_{\bm{0}}^{\bm{e}h + \bm{c}h}O(h^{s+1})\, dt.
\end{split}
\end{equation*}
This and the relation $\int_{\bm{0}}^{\bm{e}h + \bm{c}h}O(h^{s+1})\, dt = O(h^{s+2})$ imply
\begin{equation}
\label{m20d}
\bm{Y}'_{n+1} - y'(\bm{e}h+\bm{c}h) = R'(\bm{e}h + \bm{c}h) = h^{s+1}\int_{\bm{0}}^{\bm{e} + \bm{c}}g(\xi h)\prod_{i=1}^s(\xi-c_i)\, d\xi + O(h^{s+2})= O(h^{s+1}).
\end{equation}

Using the relation $R(0)=0$ and equation \eqref{eqf15}, one obtains
\begin{equation}
\label{m20b}
R(x)=\int_0^x\int_0^t R''(\xi )\,d\xi dt =\int_0^x\int_0^t\bigg{(}g(\xi)\prod_{i=1}^s(\xi -c_ih)+ O(h^{s+1})\bigg{)}\, d\xi dt .
\end{equation}
Let $x:=\omega h$ in equation \eqref{m20b} and use substitutions to get
\begin{equation}
\label{intererror}
\begin{split}
u(\omega h)-y(\omega h) = R(\omega h) &= \int_0^{\omega h}\int_0^t g(\xi)\prod_{i=1}^s(\xi -c_ih)\,d\xi dt +\int_0^{\omega h}\int_0^t O(h^{s+1})\, d\xi dt\\
&= h^{s+2}\int_0^\omega\int_0^t g(\xi h)\prod_{i=1}^s(\xi -c_i)\,d\xi dt +\int_0^{\omega h}\int_0^t O(h^{s+1})\, d\xi dt.
\end{split}
\end{equation}
It follows from equation \eqref{intererror} and the relation $\int_0^{\omega h}\int_0^t O(h^{s+1})\, d\xi dt= O(h^{s+3})$ that
\begin{align}
\label{eq14b}
\bm{Y}_{n+1}-y(\bm{e}h+\bm{c}h)
&=u(\bm{e}h+\bm{c}h)-y(\bm{e}h+\bm{c}h)=O(h^{s+2}),\\
\label{eq14c}
y_{n+1}-y(h)
&=u(h)-y(h)=O(h^{s+2}).
\end{align}
From equations \eqref{ae6}, \eqref{m20d}, \eqref{eq14b}, \eqref{eq14c}, and Definition \ref{def4.1} one concludes that 
the method has step order $p=s$ and stage order $r=s$. 

Theorem \ref{genorder} is proved. 
\end{proof}

\begin{rem}{\rm 
\label{remarka4}
It follows from equations \eqref{eq14b} and \eqref{eq14c} that
\begin{equation*}
y_{n} - y(t_n) = O(h^{s+2}),\qquad \bm{Y}_n - y(\bm{e}t_n +\bm{c}h) = O(h^{s+2}),\qquad n=1,2,...\,.
\end{equation*}
If $\bm{Y}_0$ is computed with an accuracy of $O(h^{s+2})$ at the initial step, then we have
\begin{equation}
\label{aex1}
y_{n} - y(t_n) = O(h^{s+2}),\qquad \bm{Y}_n - y(\bm{e}t_n +\bm{c}h) = O(h^{s+2}),\qquad n=0,1,...\,.
\end{equation}
Consider the following equation
\begin{equation}
\label{feq18a}
\bm{Y}'_n - y'(\bm{e}t_n  + h\bm{c}) = C_nh^{s+1}\int_{\bm{0}}^{\bm{e} +\bm{c}}  \prod_{i=1}^s(\xi-c_i)\, d\xi + O(h^{s+2}).
\end{equation}
When $n = 0$, this equation can be obtained by computing $\bm{Y}'_n$ at order of accuracy $s+1$. Also, when $n=0$, the constant $C_0$ can be set to 0 and the error is absorbed by $O(h^{s+2})$. It follows from equation \eqref{m20d} that equation \eqref{feq18a} holds for $n=1$. By induction and the arguments in the proof of Theorem \ref{genorder}, one concludes that equation \eqref{feq18a} holds true for all $n\ge 0$. 
}
\end{rem}

\subsection{Superconvergence}

\begin{thm}\label{super1}
If the collocation parameters $(c_i)_{i = 1}^s$ satisfy the equation
\begin{equation}
\label{eq16a}
\int_{0}^1\prod_{i=1}^s(\xi-c_i)\,d\xi  = 0,
\end{equation}
then the corresponding GEPTRKN method has step order of accuracy $p=s+1$. 
\end{thm}

\begin{proof}We proved in Theorem \ref{genorder} that the stage order of an $s$-stage GEPTRKN method is $r=s$ for any set of $(c_i)_{i=1}^s$. 
Let us show that if \eqref{eq16a} holds, then the step order of the corresponding method is $p=s+1$. 

From equation \eqref{eq16} and the relation $\int_0^h O(h^{s+1})\, dt = O(h^{s+2})$ one obtains
$$
y'_{n+1}-y'(t_{n+1}) = u'(t_{n+1})-y'(t_{n+1})=R'(h)=h^{s+1}\int_0^1g(\xi h)\prod_{i=1}^s(\xi -c_i)\,d\xi + O(h^{s+2}).
$$
By using the Taylor expansion $g(\xi h)=g(0)+O(\xi h)$ and orthogonality condition \eqref{eq16a}, one obtains
\begin{equation}
\label{eq16c}
y'_{n+1} -y'(t_{n+1})=h^{s+1}g(0)\int_0^1\prod_{i=1}^s(\xi -c_i)\,d\xi + O(h^{s+2}) = O(h^{s+2}).
\end{equation}
From Definition \ref{def4.1} and equations \eqref{eq14c} and \eqref{eq16c} one concludes that the stage order of the method is $p=s+1$. 
Theorem \ref{super1} is proved.  
\end{proof}

\begin{thm}\label{super2}
If an $s$-stage GEPTRKN 
method is based on collocation parameters $(c_i)_{i = 1}^s$ satisfying
\begin{equation}
\label{eq17}
\int_{0}^{1}\int_0^{1+t}\prod_{i=1}^s(\xi-c_i)\,d\xi dt = 0,\qquad
\int_{0}^1 \xi^k\prod_{i=1}^s(\xi-c_i)\, d\xi = 0,\qquad k=0,1,
\end{equation}
then the method has step order $p=s+2$. 
\end{thm}

\begin{proof}
It follows from equations \eqref{aex1} and \eqref{feq18a} and Taylor expansions of $f(t,y,z)$ 
with respect to $y$ and $z$ that
\begin{equation}
\label{mq2}
f(\bm{c}h,\bm{Y}_{n},\bm{Y}'_{n})-f(\bm{c}h,y(\bm{c}h),y'(\bm{c}h))= Ch^{s+1}\int_{\bm{0}}^{\bm{e} +\bm{c}}  \prod_{i=1}^s(\xi-c_i)\, d\xi + O(h^{s+2}).
\end{equation}
Again, if $R(t):=u(t) - y(t)$, then $R(0)=0$ and $R'(0)=0$. This and equation \eqref{mq2} imply
\begin{align*}
R''(c_i h) = u''(c_i h) - y''(c_i h) &= f(c_ih, Y_{n,i},Y'_{n,i})-f(c_ih,y(c_i h),y'(c_i h))\\
&= Ch^{s+1}\int_0^{1+c_i}\prod_{i=1}^s(\xi-c_i)\, d\xi + O(h^{s+2}), \qquad i=1,...,s.
\end{align*} 
There exists a sufficiently smooth function $w(t)$ such that $w(c_ih)=0$, $i=1,...,s$, and
\begin{equation}
\label{m21c}
R''(t)=w(t) +Ch^{s+1}\int_0^{1+t/h}\prod_{i=1}^s(\xi-c_i)\, d\xi + O(h^{s+2}).
\end{equation}
 Again, since $w(c_ih)=0$, $i=1,...,s$, there exists 
a sufficiently smooth function $g(t)$ such that $w(t)=g(t)\prod_{i=1}^s(t-c_ih)$. 
This and equation \eqref{m21c} imply
\begin{equation}
\label{aex2}
R''(t) = g(t)\prod_{i=1}^s(t-c_ih) +Ch^{s+1}\int_0^{1+t/h}\prod_{i=1}^s(\xi-c_i)\, d\xi + O(h^{s+2}).
\end{equation}
From equation \eqref{aex2} and the relation $R'(0)=0$ one gets
\begin{equation}
\label{ae2}
\begin{split}
R'(x)=\int_0^xR''(t)\,dt = &\int_0^x
g(t)\prod_{i=1}^s(t -c_ih) \,dt + \int_0^x O(h^{s+2})\,dt\\
& + Ch^{s+1} \int_0^x \int_0^{1+t/h}\prod_{i=1}^s(\xi-c_i)\, d\xi\, dt.
\end{split}
\end{equation}
From equation \eqref{ae2} with $x:=h$ and the substitution $t$ by $\xi h$ in the first integral, one derives
\begin{equation}
\label{intererror2}
\begin{split}
u'(h) - y'(h)=R'(h) =& h^{s+1}\int_{0}^{1}
g(\xi h)\prod_{i=1}^s(\xi -c_i)\,d\xi + \int_0^{h}O(h^{s+2})\, dt\\
& + Ch^{s+1} \int_0^h \int_0^{1+t/h}\prod_{i=1}^s(\xi-c_i)\, d\xi\, dt.
\end{split}
\end{equation}
Use the Taylor expansion $g(\xi h)=g(0)+\xi hg'(0)+O(h^2)$ and the relation $\int_0^{h}O(h^{s+2})\, dt = O(h^{s+3})$ and then  substitute $t = hu$ in the last integral in \eqref{intererror2} to get
\begin{equation}
\label{m20e}
\begin{split}
u'(h)-y'(h)=&h^{s+1}g(0)\int_0^{1}
\prod_{i=1}^s(\xi -c_i)\,d\xi +h^{s+2}g'(0)
\int_0^{1}\xi \prod_{i=1}^s(\xi -c_i)\,d\xi  + O(h^{s+3})\\
&+ Ch^{s+2} \int_0^1 \int_0^{1+u}\prod_{i=1}^s(\xi-c_i)\, d\xi\, du.
\end{split}
\end{equation}
Equation \eqref{m20e} and orthogonality condition \eqref{eq17} imply
\begin{equation}
\label{ae4}
y'_{n+1} - y'(t_{n+1}) = u'(h)-y'(h) = O(h^{s+3}). 
\end{equation}

From equation \eqref{ae2} and the relation $R(0) = 0$, one gets
\begin{equation}
\label{eq19}
\begin{split}
R(x) =\int_0^x R'(t)\, dt= \int_0^x\int_0^t R''(\xi) \,d\xi dt=& \int_0^x\int_0^t\bigg{(}
g(\xi)\prod_{i=1}^s(\xi -c_ih)+O(h^{s+2})\bigg{)}\,d\xi dt\\
 &+ Ch^{s+1} \int_0^x\int_0^u \int_0^{1+t/h}\prod_{i=1}^s(\xi-c_i)\, d\xi\, dt\, du.
\end{split}
\end{equation}
From equation \eqref{eq19} with $x=h$, the relation $\int_0^h\int_0^t O(h^{s+2})\, d\xi dt = O(h^{s+4})$, 
the substitutions $t$ by $th$ and $\xi$ by $\xi h$ in the first integral, and the substitutions $u$ by $uh$ and $t$ by $th$ in the second integral, one gets
\begin{equation*}
\begin{split}
y_{n+1} - y (t_{n+1}) =R(h) =& h^{s+2}\int_0^1\int_0^t g(\xi h)\prod_{i=1}^s(\xi -c_i)\,d\xi dt +O(h^{s+4})\\
& + Ch^{s+3} \int_0^1\int_0^u \int_0^{1+t}\prod_{i=1}^s(\xi-c_i)\, d\xi\, dt\, du\\
=& h^{s+2}\int_0^1\int_0^t g(\xi h)\prod_{i=1}^s(\xi -c_i)\,d\xi dt  +O(h^{s+3}).
\end{split}
\end{equation*}
This, Fubini's Theorem, the Taylor expansion $g(\xi h) = g(0) + O(h)$, and orthogonality condition \eqref{eq17} imply
\begin{equation}
\label{ae7x}
\begin{split}
y_{n+1} - y (t_{n+1}) =R(h)
=& h^{s+2}\int_0^1\int_\xi^1 g(\xi h)\prod_{i=1}^s(\xi -c_i)\,dt d\xi +O(h^{s+3})\\
=& h^{s+2}\int_0^1 g(\xi h)(1-\xi)\prod_{i=1}^s(\xi -c_i)\, d\xi +O(h^{s+3}) \\
=& h^{s+2}g(0)\bigg(\int_0^1 \prod_{i=1}^s(\xi -c_i)\, d\xi - \int_0^1 \xi\prod_{i=1}^s(\xi -c_i)\, d\xi\bigg) +O(h^{s+3})  \\
=& O(h^{s+3}).
\end{split}
\end{equation}
From \eqref{ae4}, \eqref{ae7x}, and Definition \ref{def4.1}, one concludes that the corresponding GEPTRKN method has step order $p=s+2$. 

Theorem \ref{super2} is proved. 
\end{proof}


\subsection{Stability}
\label{subsec:stability}

Applying a GEPTRKN method with coefficients $(\bm{c},\bm{A},\bm{B},\bm{b},\bm{d})$ to the test equation
\begin{equation}
\label{m20f}
y'' = \mu y' + \lambda y,\qquad y(0) = y_0,\qquad y'(0) = y'_0,\qquad \lambda \le 0,\quad \mu \le 0,\qquad (\mu,\lambda)\not=(0,0),
\end{equation}
one gets
\begin{equation}
\label{v49}
\begin{split}
y_{n+1} &= y_n + hy'_n+  h^2\bm{b}^T(\mu\bm{Y}'_n + \lambda\bm{Y}_n),\\
y'_{n+1} &= y'_n + h\bm{d}^T(\mu\bm{Y}'_n + \lambda\bm{Y}_n),\\
\bm{Y}_{n+1} &= \bm{e}y_{n+1} + \bm{c}hy'_{n+1}+ h^2\bm{A}(\mu\bm{Y}'_n + \lambda\bm{Y}_n),\\
\bm{Y}'_{n+1} &= \bm{e}y'_{n+1} + h\bm{B}(\mu\bm{Y}'_n + \lambda\bm{Y}_n).
\end{split}
\end{equation}
The characteristic equation for equation \eqref{m20f} is
$$
x^2 - \mu x - \lambda = 0
$$
whose solutions are
\begin{equation}
\label{realpart1}
x_{1,2} = \frac{\mu \pm \sqrt{\mu^2 + 4\lambda}}{2}. 
\end{equation}
One can prove that $\max(\real(x_1),\real(x_2))>0$ if either $\lambda >0$ or $\mu>0$. Here, $\real(x_i)$ denotes the real part of $x_i$, $i=1,2$. 
Thus, the solution to equation \eqref{m20f} blows up 
to infinity if either $\lambda >0$ or $\mu>0$, in general. 
The solution to equation \eqref{m20f} also blows to infinity if $\mu=\lambda = 0$, in general. 
Contour plots of $\max(\real(x_1),\real(x_2))$ are included in Figure \ref{maxrealparts}. 
One can see from the figure that the maximum of real parts of $x_{1,2}$ is greater than zero if either $\lambda >0$ or $\mu>0$.
\begin{figure}
\centerline{%
\includegraphics[scale=0.51]{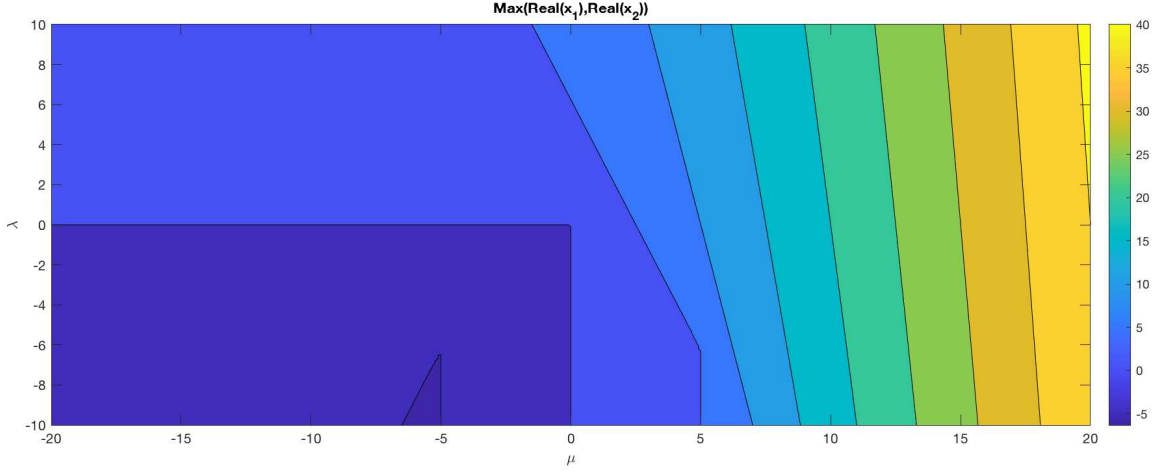}}
\caption{Contour plots of maximum of real parts of $x_{1,2}=\frac{\mu \pm \sqrt{\mu^2 + 4\lambda}}{2}$ in \eqref{realpart1}.}
\label{maxrealparts}
\end{figure}

Substitute $y_{n+1}$ and $y'_{n+1}$ from the first two equations in \eqref{v49} 
into the third equation in \eqref{v49} to get
\begin{equation}
\label{m7a}
\begin{split}
\bm{Y}_{n+1} =&\, \bm{e} \big[y_n + hy'_n+  h^2\bm{b}^T(\mu\bm{Y}'_n + \lambda\bm{Y}_n)\big] + \bm{c}h\big[y'_n + h\bm{d}^T(\mu\bm{Y}'_n + \lambda\bm{Y}_n)\big]+ h^2\bm{A}(\mu\bm{Y}'_n + \lambda\bm{Y}_n)\\
=&\, \bm{e} y_n + \big(\bm{e} + \bm{c}\big)hy'_n + h^2\lambda\big( \bm{e}\bm{b}^T + \bm{c}\bm{d}^T +  \bm{A}\big)\bm{Y}_n + h\mu \big(\bm{e}\bm{b}^T + \bm{c}\bm{d}^T +  \bm{A}\big)h\bm{Y}'_n.
\end{split}
\end{equation}
Similarly, by substituting $y'_{n+1}$ from the second equation in \eqref{v49} into the last equation in \eqref{v49} we get 
\begin{equation}
\label{m7b}
\begin{split}
h\bm{Y}'_{n+1} =&\, h\bm{e} \big[y'_n + h\bm{d}^T(\mu\bm{Y}'_n + \lambda\bm{Y}_n)\big]+ h^2\bm{B}(\mu\bm{Y}'_n + \lambda\bm{Y}_n)\\
=&\,\bm{e} hy'_n + h^2\lambda\big( \bm{e}\bm{d}^T +  \bm{B}\big)\bm{Y}_n + h\mu \big(\bm{e}\bm{d}^T +  \bm{B}\big)h\bm{Y}'_n.
\end{split}
\end{equation}
From equations \eqref{v49}, \eqref{m7a}, and \eqref{m7b} one obtains
\begin{equation*}
\label{v49.1}
\begin{bmatrix}
\bm{Y}_{n+1}\\
h\bm{Y}'_{n+1}\\
y_{n+1}\\
hy'_{n+1}
\end{bmatrix}
= \bm{M}(z,\nu)
\begin{bmatrix}
\bm{Y}_{n}\\
h\bm{Y}'_{n}\\
y_{n}\\
hy'_{n}
\end{bmatrix},\qquad z = \lambda h^2,\quad \nu = \mu h.
\end{equation*}
where
\begin{equation}
\label{f20f}
\bm{M}(z,\nu)=
\begin{bmatrix}
z\big(\bm{e}\bm{b}^T + \bm{c} \bm{d}^T + \bm{A}\big)& \nu( \bm{e}\bm{b}^T + \bm{c}\bm{d}^T + \bm{A})	&\bm{e} & \bm{e}+\bm{c}\\
z(\bm{e}\bm{d}^T + \bm{B}) &\nu(\bm{e}\bm{d}^T + \bm{B}) &\bm{0} &\bm{e}\\
z\bm{b}^T &\nu \bm{b}^T	&1 & 1\\
z\bm{d}^T  &\nu \bm{d}^T	&0 & 1\\
\end{bmatrix}.
\end{equation}

Similar to the case with classical EPTRKN methods \cite{cong99}, the stability region of GEPTRKN methods are defined as follows:

\begin{defn}
The stability region
of an $s$-stage GEPTRKN method is defined as
\begin{equation}
\label{stabregion}
S := \{(z,\nu)\not=(0,0)|\, z\le 0,\, \nu\le 0, \rho(\bm{M}(z,\nu))\leq 1\}
\end{equation}
where $\rho(\bm{M}(z,\nu))$ denotes the spectral radius of $\bm{M}(z,\nu)$ defined by equation \eqref{f20f}. 
\end{defn}

Once $(c_i)_{i=1}^s$ is chosen, the stability region of the corresponding $s$-stage GEPTRKN method can easily be studied numerically by computing the spectral radius of $\bm{M}(z,\nu)$. Stability regions of some GEPTRKN methods are investigated in Section \ref{sec:stabplot}.

\section{Extensions}
\label{sec:extensions}

We now discuss some aspects that are important to developing competitive
numerical codes.

\subsection{Variable stepsize}
\label{variablestepsize}

When the step-size $h_n$ is accepted and the next step-size $h_{n+1}$ 
is suggested, 
the values in the next step are computed by
\begin{equation}
\label{variGEPTRKN}
\begin{split}
\bm{Y}_{n+1} &= \bm{e}y_{n+1} + h_{n+1}\bm{c}y'_{n+1}+
h_{n+1}^2\bm{A}(q)f(\bm{e}t_{n}+\bm{c}h_n,\bm{Y}_{n},\bm{Y}'_{n}),\qquad q=h_{n+1}/h_n,\\
\bm{Y}'_{n+1} &= \bm{e}y'_{n+1}+
h_{n+1}\bm{B}(q)f(\bm{e}t_{n}+\bm{c}h_n,\bm{Y}_{n},\bm{Y}'_{n}),\\
y_{n+2} &= y_{n+1} + h_{n+1}y'_{n+1} +h^2_{n+1}\bm{b}^Tf(\bm{e}t_{n+1}+\bm{c}h_{n+1},\bm{Y}_{n+1},\bm{Y}'_{n+1}),\\
y'_{n+2} &= y'_{n+1}  +h_{n+1}\bm{d}^Tf(\bm{e}t_{n+1}+\bm{c}h_{n+1},\bm{Y}_{n+1},\bm{Y}'_{n+1}).
\end{split}
\end{equation}
The coefficients
 $\bm{A}(q)$ and $\bm{B}(q)$ in the first two equations in \eqref{variGEPTRKN} are 
obtained by solving the systems
\begin{equation}
\label{f20a}
\begin{split}
u(\bm{e}h_n + \bm{c}h_{n+1}) - \bm{e}u(h_n) - \bm{c}h_{n+1}u'(h_n) = h^2_{n+1}\bm{A}(q)u''(\bm{c}h_n),\qquad u = t^k,\quad k = 2,...,s+1,\\
u'(\bm{e}h_n + \bm{c}h_{n+1}) - \bm{e}u'(h_n) = h_{n+1}\bm{B}(q)u''(\bm{c}h_n),\qquad u = t^k,\quad k = 2,...,s+1.
\end{split}
\end{equation}
From the equations in \eqref{f20a}, one gets
\begin{align}
\label{f20b}
(\bm{e}h_n + \bm{c}h_{n+1})^k - \bm{e}(h_n)^k - \bm{c}h_{n+1}k(h_n)^{k-1} &= h^2_{n+1}\bm{A}(q)k(k-1)(\bm{c}h_n)^{k-2},\\
\label{f20c}
(\bm{e}h_n + \bm{c}h_{n+1})^{k-1} - \bm{e}(h_n)^{k-1}& = h_{n+1}\bm{B}(q)(k-1)(\bm{c}h_n)^{k-2}.
\end{align}
Equations \eqref{f20b} and \eqref{f20c} can be rewritten as
\begin{align}
\label{f20d}
(\bm{e} + \bm{c}q)^k - \bm{e} - k\bm{c}q &= q^2\bm{A}(q)k(k-1)\bm{c}^{k-2},\qquad k=2,...,s+1,\\
\label{f20e}
(\bm{e} + \bm{c}q)^{k-1} - \bm{e} &= q\bm{B}(q)(k-1)\bm{c}^{k-2},\qquad k=2,...,s+1.
\end{align}
Equation \eqref{f20d} can be written in vector form as follows
\begin{equation}
\label{m27a}
\bm{U} = \bm{A}(q)\begin{bmatrix} \bm{e}& \bm{c}&\cdots&\bm{c}^{s-1}\end{bmatrix}
\end{equation}
where
$$
\bm{U} = \frac{1}{q^2}\begin{bmatrix} \frac{(\bm{e} + \bm{c}q)^2 - \bm{e} - 2\bm{c}q}{2\times1}& \frac{(\bm{e} + \bm{c}q)^3 - \bm{e} - 3\bm{c}q}{3\times2}&\cdots&\frac{(\bm{e} + \bm{c}q)^{s+1} - \bm{e} - (s+1)\bm{c}q}{(s+1)\times s}\end{bmatrix}.
$$
Similarly, equation \eqref{f20e} can be written as
\begin{equation}
\label{m21b}
\bm{V} = \bm{B}(q)\begin{bmatrix} \bm{e}& \bm{c}&\cdots&\bm{c}^{s-1}\end{bmatrix}
\end{equation}
where
$$
\bm{V} = \frac{1}{q}\begin{bmatrix} \frac{(\bm{e} + \bm{c}q)^1 - \bm{e}}{1}& \frac{(\bm{e} + \bm{c}q)^2 - \bm{e}}{2}&\cdots&\frac{(\bm{e} + \bm{c}q)^{s} - \bm{e}}{s}\end{bmatrix}.
$$
The $i$-th column of the matrix $\bm{V}$ in the equation above is
\begin{equation}
\label{m21a}
\frac{(\bm{e} + \bm{c}q)^i - \bm{e}}{iq} = \frac{1}{iq}\sum_{k=1}^i{i\choose k}(\bm{c}q)^k = \sum_{k=1}^i \frac{1}{i}{i\choose k}\bm{c}^kq^{k-1}. 
\end{equation}
Denote 
$$
\bm{W} := \begin{bmatrix} \bm{c}& \bm{c}^2&\cdots&\bm{c}^{s}\end{bmatrix},\qquad \bm{\Gamma} :=  \begin{bmatrix} \bm{\alpha}_1& \bm{\alpha}_2&\cdots&\bm{\alpha}_s\end{bmatrix}
$$
where
$$
\bm{\alpha}_i = \bigg[\frac{1}{i}{i\choose 1}\quad \frac{1}{i}{i\choose 2} \, \cdots \, \frac{1}{i}{i\choose i} \quad 0 \, \cdots \, 0\bigg]^T \in \mathbb{R}^s,\qquad 1\le i\le s.
$$
Then equation \eqref{m21a} can be rewritten as
$$
\frac{(\bm{e} + \bm{c}q)^i - \bm{e}}{iq} = \bm{W} \Diag(1,q,...,q^{s-1})\bm{\alpha}_i,\qquad 1\le i\le s.
$$
Thus, $\bm{V} = \bm{W}\Diag(1,q,...,q^{s-1})\bm{\Gamma}$ and equation \eqref{m21b} can be written as
$$
 \bm{W}\Diag(1,q,...,q^{s-1})\bm{\Gamma} \begin{bmatrix} \bm{e}& \bm{c}&\cdots&\bm{c}^{s-1}\end{bmatrix}^{-1}= \bm{B}(q).
$$
This means that the matrix $\bm{B}(q)$ can be updated by a diagonal scaling. The same is true for the matrix $\bm{A}(q)$ in equation \eqref{m27a}.

The approximations $y_{n+2} \approx y(t_{n+1}+h_{n+1})$ and $y'_{n+2} \approx y'(t_{n+1}+h_{n+1})$ obtained from 
the new suggested step-size $h_{n+1}$ are subject to a local truncation error denoted by LTE 
which is often computed by using another embedded method (see 
Section~\ref{subsec:embedded} below). If the estimated error LTE
is smaller than a prescribed tolerance TOL, then $h_{n+1}$ is accepted. 
Otherwise, it is rejected and a reduced step-size $\tilde{h}_{n+1}$ is
suggested to recompute $\bm{Y}_{n+1}\approx 
y(\bm{e}t_{n+1}+\bm{c}\tilde{h}_{n+1})$, $\bm{Y}'_{n+1}\approx 
y'(\bm{e}t_{n+1}+\bm{c}\tilde{h}_{n+1})$, $y'_{n+2}\approx 
y'(t_{n+1}+\tilde{h}_{n+1})$, and $y_{n+2}\approx 
y(t_{n+1}+\tilde{h}_{n+1})$.
This process is repeated until an accepted value of $h_{n+1}$ is found.

Using a variable step-size from a collocation perspective means that, 
from the past accepted values $y_{n}$, $y'_{n}$, $\bm{Y}_n$, and $\bm{Y}'_n$ we construct the 
collocation solution $u(t)$ defined in \eqref{interpol} and then evaluate and store the values $y_{n+1}=u(t_n+h_n)$, $y'_{n+1}=u'(t_n+h_n)$.  
The acceptance of $h_{n+1}$ is subject to a local truncation error for computing $y_{n+2}$ using this step-size (cf. Section \ref{errorcontrol}). 
Therefore, the collocation solution $u(t)$ and values $y_{n+1}$, $y'_{n+1}$ remain the same when $h_{n+1}$ varies. We only
adjust the step-size $h_{n+1}$ for computing $\bm{Y}_{n+1}$ and $\bm{Y}'_{n+1}$ for the next step.
Consequently, the following generalization 
of~\cite[Theorem~2.1]{cong01} on order of accuracy of GEPTRKN remains valid.

\begin{thm}
The $s$-stage variable step-size
GEPTRKN method \eqref{variGEPTRKN} is of stage order $r=s$ and of step order at 
least $p=s$ for any set of distinct collocation points $(c_i)_{i=1}^s$. It has step order $p=s+1$, if the parameters $(c_i)_{i=1}^s$ satisfy
the orthogonality conditions
$$
\int_0^1\prod_{i=1}^s(\xi-c_i)d\xi = 0.
$$
\end{thm}


\subsection{Interpolation}
\label{interpolation}

A continuous extension of an $s$-stage GEPTRKN method determined by $(\bm{c},\bm{A},\bm{B},\bm{b},\bm{d})$ is defined  
as follows (cf. \cite{cong01}, \cite{nguyenfeptrk})
\begin{align}
\label{denseoutput}
y_{n+\xi} &= y_{n} + \xi h_ny'_n+ (\xi h_n)^2\bm{b}^T(\xi)
f(\bm{e}t_{n}+\bm{c}h,\bm{Y}_{n},\bm{Y}'_{n}),\\
\label{denseoutput1}
y'_{n+\xi} &= y'_{n} + \xi h_n\bm{d}^T(\xi)
f(\bm{e}t_{n}+\bm{c}h,\bm{Y}_{n},\bm{Y}'_{n}),\qquad 0\leq \xi \leq 1,
\end{align}
where $y_{n+\xi}\approx y(t_{n+\xi})$ and $y'_{n+\xi}\approx y'(t_{n+\xi})$ with 
$t_{n+\xi}:=t_n+\xi h_n$, and the coefficients 
$\bm{b}(\xi)$ and $\bm{d}(\xi)$ are obtained from the equations
\begin{equation*}
\begin{split}
(t_n + \xi h_n)^k &= (t_n)^k +\xi h_n k(t_n)^{k-1} + (\xi h_n)^2 \bm{b}^T(\xi)k(k-1)(\bm{e}t_n + \bm{c}h_n)^{k-2},\qquad k=2,...,s+1,\\
(t_n + \xi h_n)^{k} &= (t_n)^{k} + \xi h_n \bm{d}^T(\xi)k(\bm{e}t_n + \bm{c}h_n)^{k-1},\qquad k=1,...,s.
\end{split}
\end{equation*}
These equations are simplified to
\begin{equation}
\label{m27b}
\begin{split}
\xi^{k-2} &=  \bm{b}^T(\xi)k(k-1)\bm{c}^{k-2},\qquad k=2,...,s+1,\\
\xi^{k-1} &= \bm{d}^T(\xi)k\bm{c}^{k-1},\qquad k=1,...,s.
\end{split}
\end{equation}
The equations in \eqref{m27b} can be rewritten in vector form as follows
\begin{equation}
\label{m22a}
\begin{split}
\begin{bmatrix}
\frac{1}{2} & \frac{\xi}{3\times 2} &\cdots &\frac{\xi^{s-1}}{(s+1)s}
\end{bmatrix}
&= \bm{b}^T(\xi)\begin{bmatrix}\bm{e} &\bm{c} &\bm{c}^2 &\cdots &\bm{c}^{s-1}\end{bmatrix},\\
\begin{bmatrix}
\frac{1}{1} & \frac{\xi}{2} &\cdots &\frac{\xi^{s-1}}{s}
\end{bmatrix}
&= \bm{d}^T(\xi)\begin{bmatrix}\bm{e} &\bm{c} &\bm{c}^2 &\cdots &\bm{c}^{s-1}\end{bmatrix}.
\end{split}
\end{equation}
From their definition we have 
$y_{n+\xi}=u(t_n+\xi h_n)$ and $y'_{n+\xi}=u'(t_n+\xi h_n)$, where 
$u(t)$ is the collocation solution defined in 
\eqref{interpol}. So technically, $y_{n+\xi}$ and $y'_{n+\xi}$ are obtained as if $t_{n+\xi}=t_n+\xi h_n$
was the end point.
From \eqref{intererror} it is easy to check that
$y_{n+\xi}-y(t_n+\xi h_n)=O(h_n^{s+2})$. 
From equation \eqref{ae2} and similar substitutions as in equation \eqref{intererror2} one derives
$y'_{n+\xi}-y'(t_n+\xi h_n)=O(h_n^{s+1})$, $\forall \xi \in [0,1]$.
Hence, the following result holds.

\begin{thm} The GEPTRKN method defined by \eqref{denseoutput}--\eqref{denseoutput1} with 
$\bm{b}(\xi)$ and $\bm{d}(\xi)$ defined by \eqref{m22a} produces a continuous 
GEPTRKN method of order $s$, i.e.,
$$
y(t_n+\xi h_n)-y_{n+\xi}=O(h_n^{s+1}),\quad y'(t_n+\xi h_n)-y'_{n+\xi}=O(h_n^{s+1}),\qquad 0 \le \xi \le 1.
$$
\end{thm}


\subsection{Embedded methods}
\label{subsec:embedded}

\def\tc{\tilde{c}}
\def\ts{{\tilde{s}}}
\def\tY{\tilde{Y}}
\def\ty{\tilde{y}}
\def\tu{\tilde{u}}
\def\tbm#1{\tilde{\bm{#1}}}

Consider an $s$-stage GEPTRKN method with coefficients $(\bm{c},\bm{A},\bm{B},\bm{b},\bm{d})$.
We will construct an embedded GEPTRKN method ($\tbm{c},\tbm{A},\tbm{B},\tbm{b},\tbm{d})$ paired with the GEPTRKN method $(\bm{c},\bm{A},\bm{B},\bm{b},\bm{d})$ 
to cheaply estimate the local errors and control 
the step-size in practice.

Let $\ts < s$, $\{\tc_1, ..., \tc_\ts\}
\varsubsetneq \{c_1, ..., c_s\}$,
$\tbm{c} = (\tc_1, ..., \tc_\ts)^T$
and $\tbm{e} = (1, ..., 1)^T$ of length $\ts$. 
An embedded pair GEPTRKN methods in which 
another approximation $\ty_{n+1}$ to $y(t_{n+1})$ can 
be computed without any extra right-hand side function evaluation is defined by 
\begin{equation}
\label{embpair}
\begin{split}
y_{n+1} &= y_{n} +h_ny'_n+ h_n^2\bm{b}^Tf(\bm{e}t_{n}+\bm{c}h_n,\bm{Y}_{n},\bm{Y}'_{n}),\\
y'_{n+1} &= y'_{n} + h_n\bm{d}^Tf(\bm{e}t_{n}+\bm{c}h_n,\bm{Y}_{n},\bm{Y}'_{n}),\\
\ty_{n+1} &= y_{n} +h_ny'_n + h_n^2\tbm{b}^Tf(\tbm{e}t_{n}+\tbm{c}h_n,\tbm{Y}_{n},\tbm{Y}'_{n}),\\
\ty'_{n+1} &= y'_n + h_n \tbm{d}^Tf(\tbm{e}t_{n}+\tbm{c}h_n,\tbm{Y}_{n},\tbm{Y}'_{n}),\\
\bm{Y}_{n+1} &= \bm{e}y_{n+1} +\bm{c}h_ny'_{n+1} +  h_n^2\bm{A}f(\bm{e}t_{n}+\bm{c}h_n,\bm{Y}_{n},\bm{Y}'_{n}),\\
\bm{Y}'_{n+1} &= \bm{e}y'_{n+1} +  h_n\bm{B}f(\bm{e}t_{n}+\bm{c}h_n,\bm{Y}_{n},\bm{Y}'_{n}).
\end{split}
\end{equation}
Here
$\tbm{Y}_{n} :=(\tY_{n,1},...,\tY_{n,\ts})^T$, $\tbm{Y}'_{n}:=(\tY'_{n,1},...,\tY'_{n,\ts})^T$, 
and 
$$
f(\tbm{e}t_{n}+\tbm{c}h_n,\tbm{Y}_{n},\tbm{Y}'_{n})
:= 
(f(t_n+\tc_1h_n,\tY_{n,1},\tY'_{n,1}), ..., f(t_n+\tc_\ts h_n,\tY_{n,\ts},\tY'_{n,\ts}))^T 
$$
where $\tY_{n,i}$ and $\tY'_{n,i}$ are defined by the following rule
\begin{align*}
\text{if}\quad \tc_i=c_j\quad \text{then}\quad 
\tY_{n,i}=Y_{n,j},\quad \tY'_{n,i}=Y'_{n,j}.
\end{align*}
The coefficients $\tbm{b} = (\tilde{b}_1,...,\tilde{b}_{\tilde{s}})^T$ in \eqref{embpair} are defined as the coefficients of the GEPTRKN method 
generated from the collocation parameters $(\tilde{c}_i)_{i=1}^{\tilde{s}}$ which is a subset of $(c_i)_{i=1}^s$.
The solution $\ty_{n+1}$ is computed by using the subset 
of past stage values with indices corresponding to $\tbm{c}$.
For this definition of $\tbm{b}$ we are ensured that 
$\ty_{n+1}=y(t_{n+1})+O(h^{\ts+1})$ as a result of 
Theorem~\ref{genorder}.

\begin{thm}
An $s$-stage embedded pair GEPTRKN \eqref{embpair} 
produces numerical solutions $y_{n+1}$ and $\ty_{n+1}$ that satisfy
$$
y_{n+1}-\ty_{n+1}=O(h^{\ts+1}),
$$
for all set of collocation parameters $(c_i)_{i=1}^s$.
\end{thm}


\subsection{Error control and step-size change}
\label{errorcontrol}

Let us discuss a strategy for changing step-sizes in the implementation of 
a GEPTRKN method of order $p$ embedded with a GEPTRKN method of order $\tilde{p}<p$ using the variable step-size technique in Section \ref{variablestepsize}.
 At each step we compute a local truncation error LTE as follows
\begin{equation}
\label{LTE1}
\text{LTE} = \|y_{n+1} - \tilde{y}_{n+1}\| = O(h^{\tilde{p}+1}) 
\end{equation}
where $\|\cdot\|$ denotes the 2-norm. 
We opt for formula \eqref{LTE1} for computing $\text{LTE}$ instead of using a more complicated one introduced in \cite{cong01} by the formula
\begin{equation}
\label{LTE2}
\text{LTE} = \sqrt{\|y_{n+1} - \tilde{y}_{n+1}\|^2 + \|y'_{n+1} - \tilde{y}'_{n+1}\|^2}= O(h^{\tilde{p}+1}). 
\end{equation}
Here $\tilde{y}'_{n+1}$ is computed by
$$
\ty'_{n+1} = y'_{n} + h_n\tbm{d}^Tf(\tbm{e}t_{n}+\tbm{c}h_n,\tbm{Y}_{n},\tbm{Y}'_{n}),
$$
where $\tbm{Y}_{n}$ and $\tbm{Y}'_{n}$ are defined as in Section \ref{subsec:embedded} and  ($\tbm{c},\tbm{A},\tbm{B},\tbm{b},\tbm{d})$ are the coefficients of the embedded method.  

It has been observed from our experiments that using \eqref{LTE2} instead of \eqref{LTE1} for computing $\text{LTE}$ results in having smaller step-sizes, and, therefore, yields numerical solutions of higher accuracy. However, using smaller step-sizes leads to more right-hand side function evaluations. Overall, we do not see any advantage of using \eqref{LTE2} over using \eqref{LTE1} for computing $\text{LTE}$ in terms of accuracy versus the number of right-hand side function evaluations.   

In our implementation a step-size $h_n$ is accepted if $\text{LTE}\le\text{TOL}$ and rejected if otherwise. 
If $h_n$ is rejected, then the process is repeated with the new step-size $h_n=h_n/2$ until an 
accepted $h_n$ is found. 
If $h_n$ is accepted, then the step-size $h_{n+1}$ in the next step is defined by
$$
h_{n+1} = h_n \min\bigg\{2,\max\big\{0.5,0.8\big(\text{TOL}/\text{LTE}\big)^{1/(\tilde{p}+1)}\big\}\bigg\}. 
$$
For this formula the ratio $h_{n+1}/h_n$ always stays in the interval $[0.5,2]$.  
This step-size changing technique was also used in \cite{cong01}. 

\section{Numerical experiments}

The numerical experiments in this section were conducted in double precision (machine precision = $0.2\times 10^{-15}$) using MATLAB software running on a computer with 2.2 GHz Intel Core i7 processor and 16 gb of RAM.
\subsection{Derivation of some methods}
\label{derivation}

We implement the new methods with the following sets of collocation parameters $\bm{c}$:

\begin{align*}
\bm{c}_1 &= [0.182647322580547\quad
   0.742402187612118 \quad
   1.474950489807336]^T,\\
\bm{c}_2 &= [   0.138502716885383\quad
   0.605842632479162\quad
   1\quad
   1.588987983968791]^T,\\
\bm{c}_3 &= [0\quad
   0.253662773062501\quad
   0.693421021629012\quad
   1\quad
   1.624344776737066]^T,\\
\bm{c}_4 &= [                   0\quad
   0.160867438838146\quad
   0.475690327561694\quad
   0.809991289295481\quad
   1\quad
   1.664562055415935]^T. 
 \end{align*}
 These sets of collocation parameters $\bm{c}_i$, $i=1,...,4$, are computed to satisfy orthogonality condition \eqref{eq17}.
The GEPTRKN methods based on $\bm{c}_1$, $\bm{c}_2$, $\bm{c}_3$, and $\bm{c}_4$ are of order of accuracy 5, 6, 7, and 8, by Theorem \ref{super2}, respectively. These methods will be denoted by geptrkn5, geptrkn6, geptrkn7, and geptrkn8 when implemented with fixed step-sizes. 
We also implement these methods with variable step-size technique described in Section \ref{subsec:embedded}. 
The embedded methods used with $\bm{c}_1$, $\bm{c}_2$, $\bm{c}_3$, and $\bm{c}_4$ have order of accuracy of 2, 3, 4, and 5, respectively. The variable step-size versions of geptrkn5, geptrkn6, geptrkn7, and geptrkn8 are denoted by geptrkn52, geptrkn63, geptrkn74, and geptrkn85, respectively.  
 
 We also implement a 5-stage GEPTRKN method with the following set of collocation parameters
 \begin{equation*}
\label{coef54}
\begin{split}
\bm{c} = [0.14717733121747\quad
   &0.66145426898123\quad
   1.28305172479853 \\
   &\qquad \qquad \qquad 1.81537781109684\quad
   2.25988885044222]^T.
\end{split}
\end{equation*}
This set of parameters $\bm{c}$ does not satisfy orthogonality condition \eqref{eq17}. 
The corresponding method is of order of accuracy $p=5$ while the embedded method is of order of accuracy $p=4$. 
The obtained method is denoted by geptrkn54.

\subsection{Stability regions}
\label{sec:stabplot}

Figure \ref{figstab1} presents contour plots of the spectral radius of the stability matrix $M(z,\nu)$ for the geptrkn5 (left) and geptrkn6 (right) methods. From Figure \ref{figstab1} one can conclude that the stability regions of these methods are sufficiently large for solving non-stiff equations. 

\begin{figure}
\centerline{%
\includegraphics[scale=0.6]{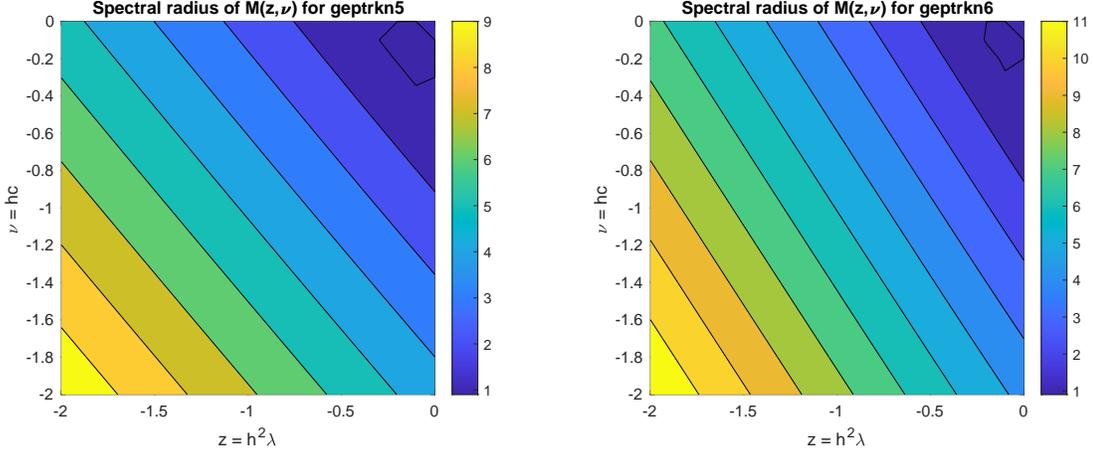}}
\caption{Contour plots of the spectral radius of the stability matrix $M(z,\nu)$ for the geptrkn5 (left) and geptrkn6 (right) methods.}
\label{figstab1}
\end{figure}

Contour plots of the spectral radius of the stability matrix $M(z,\nu)$ for the two methods geptrkn7 and geptrkn8 are provided in Figure \ref{figstab2}. One can see that the higher the order of accuracy of the method is the smaller the stability region it possesses. However, the stability regions of these methods are sufficiently large for solving non-stiff problems. 

\begin{figure}
\centerline{%
\includegraphics[scale=0.6]{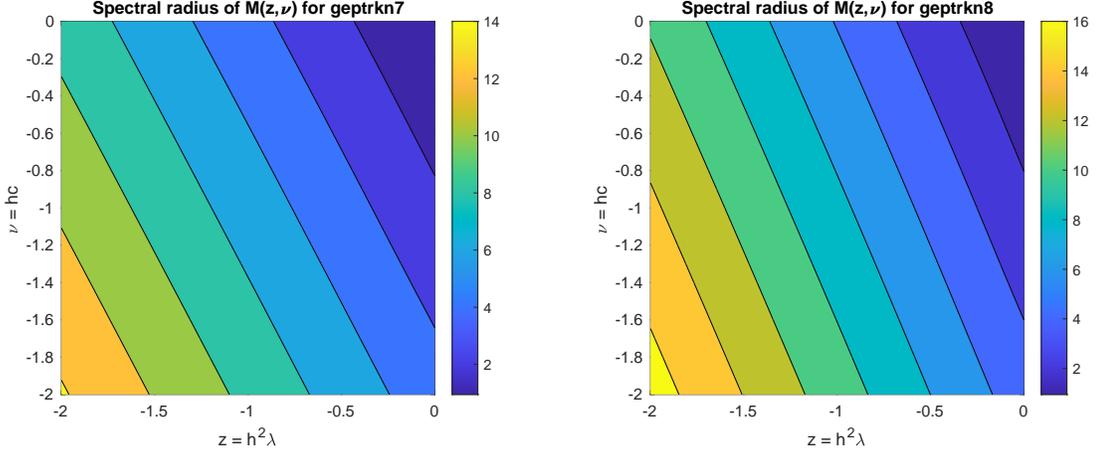}}
\caption{Contour plots of the spectral radius of the stability matrix $M(z,\nu)$ for the geptrkn7 (left) and geptrkn8 (right) methods.}
\label{figstab2}
\end{figure}

The stability region of the geptrkn54 method is similar to that of the geptrkn8 method and is not included in this paper for simplicity. 

\subsection{Test problems}

To test the performance of the new methods, we carried out 
numerical experiments with the following problems:

\begin{itemize}
\item{}
{\bf LINE} -- Consider the following linear equation
\begin{equation}
\label{testeq}
y'' = -cy' - \lambda y - 2\cos(2t) - 4\sin(2t),\quad 0\le t\le t_{end},\quad y(0) = 2,\quad y'(0) = -1.
\end{equation}
When $c = \lambda = 2$, the exact solution to equation \eqref{testeq} is $y(t) = e^{-t}\cos t + \cos(2t)$. In our experiments we used $t_{end} = 10$.
\item{}
{\bf TELE} -- Consider the telegraph equation the most well-known example of (a homogeneous version of) the general wave equation:
\begin{equation}
\label{telegraph}
a^2 u_{xx} = u_{tt} + \gamma u_t + ku.
\end{equation}
Here $u(x, t)$ is the voltage inside a piece of telegraph/transmission wire, whose electrical properties per unit length are: 
resistance $R$, inductance $L$, capacitance $C$, and conductance of leakage current $G$. The constants $a$, $\gamma$, and $k$
are defined by
$$
a^2 = \frac{1}{LC},\qquad \gamma = \frac{G}{C} + \frac{R}{L},\qquad k = \frac{GR}{CL}. 
$$
Equation \eqref{telegraph} can be rewritten as
\begin{equation}
\label{tele2}
u_{tt} = -\gamma u_t - ku + a^2 u_{xx}.
\end{equation}
We solve numerically equation \eqref{tele2} with the following boundary conditions:
$$
u(0,t) = u(1,t) = 0,\quad t\ge 0,\qquad u(x,0) = \sin(\pi x),\quad 0\le x\le 1.
$$
\item{}
{\bf VAND} -- The Van der Pol oscillator problem
\begin{align*}
y'' = \mu(1 - y^2)y' - y,\qquad y(0) = 2,\quad y'(0) = 0,\qquad \mu >0.
\end{align*}
A formula for the solution to the VAND problem is not known. 
This problem is very stiff if $\mu$ is large. However, the problem is non-stiff if $\mu$ is small and  
we use $\mu = 1$ in our experiments. 
The integration domain for this problem is $[0,10]$.
\end{itemize}

\subsection{Results and discussion}

First, we carried out numerical experiments with the four methods geptrkn5, geptrkn6, geptrkn7, and geptrkn8 to verify our super-convergence result in Theorem \ref{super2}.  By the theorem, these methods are expected to have order of accuracy of 5, 6, 7, and 8, respectively, in practice. 
In our experiments, we compute the NCD number which is defined as follows
$$
\text{NCD} = \log_{10} Error,\qquad Error:=\max_{0\le t_n\le T} |y_{i}^{\text{comput}}(t_n) - y_i(t_n)|.
$$
Recall that if a numerical method is of order of accuracy $p$, then $Error \approx Ch^p$. This implies 
$$
NCD(h) = \log_{10} Error \approx \log_{10}(Ch^p) = \log_{10}(C) + p\log_{10}h.
$$
Thus, when the step-size $h$ decreases by half, then we have 
$$
NCD(h/2) \approx \log_{10}(C) + p\log_{10}(h/2) = \log_{10}(C) + p\log_{10}h - p\log_{10}2.
$$
Therefore, when the step-size $h$ is halved, we expect the NCD values decrease by $p\log_{10}2 \approx 0.3\times p$. 

NCD values for the LINE problem generated by the four methods geptrkn5, geptrkn6, geptrkn7, and geptrkn8 
are reported in Table \ref{table1}. We omitted the NCD values from Table \ref{table1} when they reach machine precision, i.e., when $NCD\approx -15$. 
From Table \ref{table1} we can see that the NCD values for the geptrkn5 decrease almost by 1.5 = 0.3*5 when the step-size $h$ is halved. Thus, we conclude that the order of accuracy of the geptrkn5 method is 5. Similarly, when the step-size $h$ decreases by half, the NCD values of the geptrkn6 method decrease almost by 1.8 = 0.3*6. Therefore, the geptrkn6 method has order of accuracy $p=6$. The fact that the order of accuracy of the geptrkn7 and geptrkn8 methods are $p=7$ and $p=8$, respectively, are not clearly seen from the NCD values in Table \ref{table1}. A possible reason for this is: the stability regions of geptrkn7 and geptrkn8 are smaller than those of the other two methods. Due to their high order of accuracy, the NCD values of the geptrkn7 and geptrkn8 methods reach machine precision at larger step-size $h$  compared to the other methods. The conclusion from this experiment is: super-converge is obtained in practice as expected from the theoretical result in Theorem \ref{super2}.

\begin{table}[ht] 
\caption{NCD values for the LINE problem}
\label{table1}
\centering
\small
\renewcommand{\arraystretch}{1.2}
\begin{tabular}{@{}l@{}|c@{\hspace{2mm}}|c@{\hspace{2mm}}|c@{\hspace{2mm}}|c@{\hspace{2mm}}|c@{\hspace{2mm}}|
c@{\hspace{2mm}}c@{\hspace{2mm}}|}
\hline
&$h$	&geptrkn5 & geptrkn6& geptrkn7 &geptrkn8 \\
\hline
&$1/2^2$  &-1.3  & 0.2  &-0.0  & 0.6\\ 
&$1/2^3$  &-4.3  &-5.6  &-6.7  &-8.3\\ 
&$1/2^4$  &-5.7  &-7.2  &-8.6  &-10.2\\ 
&$1/2^5$  &-7.1  &-9.0  &-10.5  &-12.4\\ 
&$1/2^6$  &-8.6  &-10.7  &-12.5  &-14.6\\ 
&$1/2^7$  &-10.1  &-12.5  &-14.6  &--\\ 
&$1/2^8$  &-11.6  &-14.2  &--  &--\\ 
&$1/2^9$  &-13.1  &--  &--  &--\\ 
&$1/2^{10}$  &-14.4  &--  &--  &--\\ 
\hline
\end{tabular} 
\end{table}

In the following experiments we will compare the performance of the variable step-size implementation of the methods derived in Section \ref{derivation}, the embedded pair explicit pseudo two-step Runge-Kutta methods cong5 proposed in \cite{CPW98}, and the MATLAB function ode45 on the LINE, TELE, and VAND problems.

The errors reported in Figures \ref{figLINE}, \ref{figTELE}, and \ref{figVAND} are computed as follows
$$
\text{Error} = \sqrt{\sum_{i=1}^k\big(y^{\text{comput}}_i(t_{\text{end}}) - y_i(t_{\text{end}})\big)^2}
$$
where $k$ is the dimension of the ODE system to solve. In our experiments, $k =1$ for the LINE and VAND problems while $k = 10$ for the TELE problem.

Figure \ref{figLINE} plots the numerical errors versus the number of right-hand side function evaluations (NFE) 
for the seven methods geptrkn52, geptrkn63, geptrkn74, geptrkn85, geptrkn54, ode45, and cong5. 
From Figure \ref{figLINE} we can see that the geptrkn52 method is of order of accuracy 5 as its numerical error curve is almost parallel to that of the ode45. The geptrkn63 method has a numerical error curve with a steeper slope than those of the numerical error curves of the ode45 and the eptrkn52 methods. Thus, the geptrkn63 method has a higher order of accuracy which is $p=6$. Similarly, one can see that the geptrkn74 and geptrkn85 methods are of order of accuracy $p=7$ and $p=8$, respectively. Although the geptrkn54 method is of accuracy order $p=5$, it is the most efficient method in the experiment. This doesn't disagree with what has been seen in the literature. Specifically, among RK methods sharing the same order of accuracy, the one using the least stages does not necessarily yield the best numerical results. Also, from Figure \ref{figLINE} one can see that except for the geptrkn52 and geptrkn63 methods, all other methods are more efficient than the ode45 method. 

\begin{figure}
\centerline{%
\includegraphics[scale=0.6]{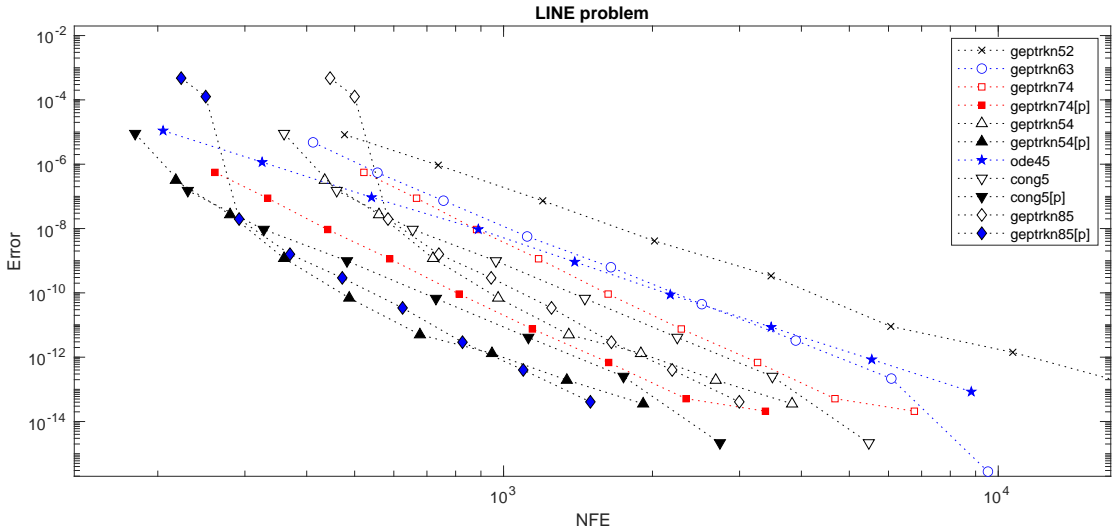}}
\caption{Numerical results for the LINE problem.}
\label{figLINE}
\end{figure}

For the TELE problem, equation \eqref{telegraph} was discretized by the method of lines by means of spectral methods in \cite{Tre}. Specifically, we used Chebyshev-Gauss-Lobatto points to discretize the interval [0, 1] and obtained a system of 10 second-order ordinary differential equations. Then we carried out numerical experiments with this system.

Figure \ref{figTELE} plots the numerical results for the TELE problem with $a = 0.01$ and $\gamma=1$. 
When high accuracy is required the 3 methods geptrkn74, geptrkn54, and geptrkn63 yield better results than do the ode45 and cong5 methods. The geptrkn54 is the best method in this experiment. The slopes of the error curves from the new methods agree with the theoretical result on order of accuracy in Theorem \ref{super2}. Namely, the error curve of the geptrkn52 method has a similar slope as that of the error curve of the ode45. In addition, if implemented in parallel computing environments, the new methods will be much better than the ode45 method. Here we assume that the speedup factor is only 2 even thought it has been observed from experiments that the speedup factor is often greater than 2 \cite{cong00}.
\begin{figure}
\centerline{%
\includegraphics[scale=0.6]{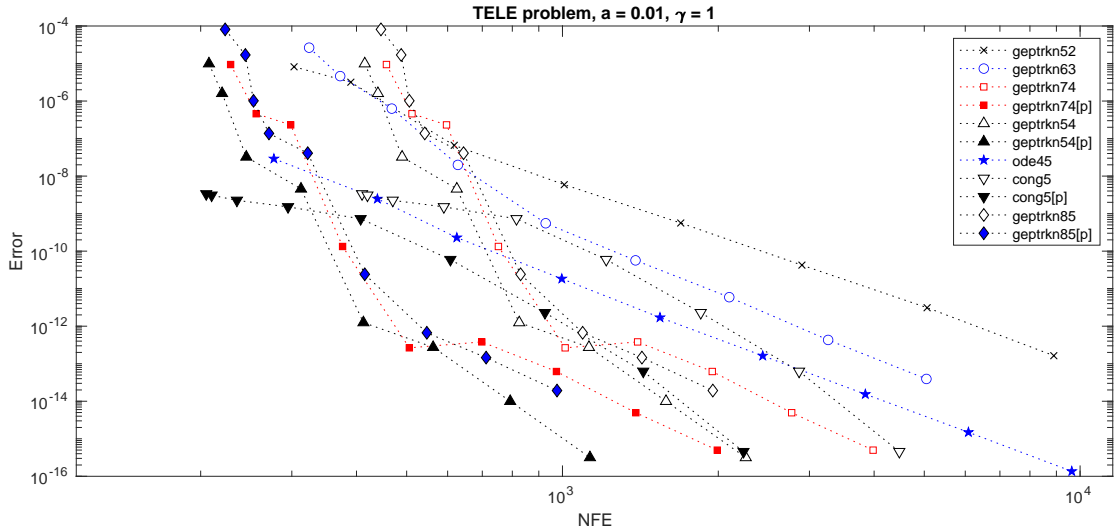}}
\caption{Numerical results for the TELE problem when $a=0.01$ and $\gamma=1$.}
\label{figTELE}
\end{figure}

Figure \ref{figVAND} presents numerical results for the VAND problem with $\mu=1$. For this value of $\mu$, the VAND problem is a non-stiff one. It is well-known that the VAND problem is a very stiff one if $\mu$ is large. Again, it can be seen from Figure \ref{figVAND} that the two methods  geptrkn85, and geptrkn54 are superior to the ode45 and cong5 methods. 

\begin{figure}
\centerline{%
\includegraphics[scale=0.6]{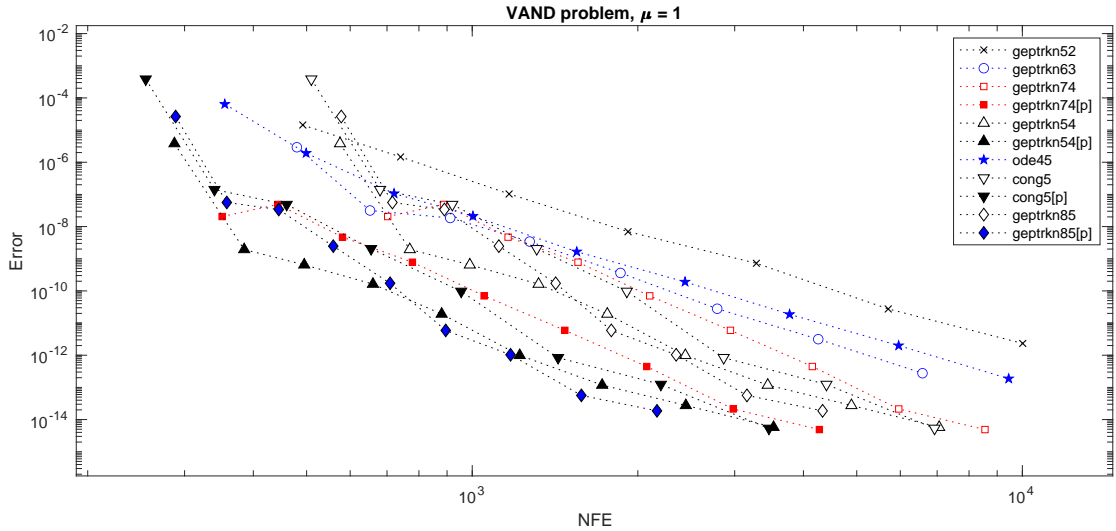}}
\caption{Numerical results for the VAND problem when $\mu=1$.}
\label{figVAND}
\end{figure}

The conclusion from the experiments above is: the new methods are superior to the ode45 and the cong5 methods for solving non-stiff problems even in sequential computing environments. If implemented in parallel computing environments, the new methods will be much more efficient. In addition, the numerical results agree with our theoretical results on order of accuracy of the new methods.

\section{Concluding remarks}

A new class of generalized explicit pseudo two-step RKN (GEPTRKN) methods has been developed and  studied in this paper. The new methods are applicable to second-order initial value problems in general form \eqref{sys}. When the first derivative $y'$ is absent from the right-hand side function $f(t,y,y')$, the new methods reduce to the classical explicit pseudo two-step RKN methods \cite{cong99}. We proved that an $s$-stage GEPTRKN method has order of accuracy $p=s$ for any set of collocation parameters $(c_i)_{i=1}^s$. When the set of collocation parameters $(c_i)_{i=1}^s$ satisfies some orthogonality conditions, the corresponding method can attain order of accuracy $p=s+2$. 
The theoretical super-convergence results in the paper have been confirmed by our numerical experiments. Numerical comparisons among the new methods, the explicit pseudo two-step RK method cong5, and the MATLAB function ode45 have shown that the new methods are more efficient for solving non-stiff second-order initial value problems. 
Since GEPTRKN methods have the structure of EPTRKN methods, they will even be much more efficient when implemented on parallel computing environments.

\section*{Declarations and statements}

\begin{itemize}
\item{}
Data sharing not applicable to this article as no datasets were generated or analyzed during the current study.
\item{}
The authors declare that they have no conflict of interest.
\end{itemize}







\begin{thebibliography}{00}

\bibitem{Burrage}
K. Burrage, {\it Parallel and sequential methods for ordinary differential equations}, Oxford University Press, Oxford, 1995.

\bibitem{Butcher} J. C. Butcher, \emph{ Numerical Methods for Ordinary Differential Equations}, Wiley, 3rd ed., 2016.



 
\bibitem{cong99} 
N. H. Cong, K. Strehmel, R. Weiner, A general class of explicit pseudo two-step RKN methods on parallel computers,
{\it Comput. Math. Appl.}, 38 (1999), 17--30. 
 
 \bibitem{congeptrk99}
N. H. Cong, Explicit pseudo two-step Runge-Kutta methods for parallel computers, {\it Int. J. Comput. Math.}, 73 (1999), 77--91.


\bibitem{cong01}
N. H. Cong, Explicit pseudo two-step RKN methods with stepsize control, {\it Appl. Numer. Math.}, 38 (2001), 135--144.

\bibitem{congminh07} N. H. Cong, N. V. Minh, Continuous parallel-iterated RKN-type PC methods for nonstiff IVPs, \emph{ Appl. Numer. Math.}, 57 (2007), 1097--1107. 

\bibitem{CPW98}
N. H. Cong, H. Podhaisky, R. Weiner, Numerical experiments with some explicit pseudo two-step RK methods on a shared memory computer, {\it Comput. Math. Appl.}, 36 (2) (1998), 107--116.

\bibitem{cong00}
N. H. Cong, H. Podhaisky, R. Weiner, Efficiency of embedded explicit pseudo two-step RKN methods on a shared memory parallel computer, {\it Vietnam J. Math.}, 34 (1) (2006), 95--108.


 
 



 




\bibitem{Franco2}
J. M. Franco,
 Exponentially fitted explicit {R}unge-{K}utta-{N}ystr\"{o}m methods, 
 {\it J. Comput. Appl. Math.}, 167 (2004), 1--19.

 
 \bibitem{Franco14A}
J. M. Franco, I. G{\'o}mez, 
Trigonometrically fitted nonlinear two-step methods for solving second order oscillatory IVPs, 
{\it Appl. Math. Comput.}, 232 (2014), 643--657.


\bibitem{Franco14J}
J. M. Franco, I. G{\'o}mez, 
Symplectic explicit methods of Runge-Kutta-Nystr\"{o}m type for solving perturbed oscillators, 
{\it J. Comput. Appl. Math.}, 260 (2014), 482--493.


\bibitem{Hairer}
E. Hairer, G. Wanner, {\it Solving Ordinary Differential Equations I: Nonstiff Problems}, Springer-Verlag, Berlin, 1991. 
\bibitem{HW} E. Hairer, G. Wanner, \emph{ Solving Ordinary Differential Equations II, Stiff and Differential-algebraic Problems}, Springer, Berlin, 1991

 
\bibitem{nguyen1} 
N. S. Hoang, R. B. Sidje, N. H. Cong, Analysis of trigonometric implicit Runge-Kutta methods, {\it J. Comput. Appl. Math.}, 198 (2007), 187--270.
 
 \bibitem{nguyen2}
 N. S. Hoang, R. B. Sidje, N. H. Cong, On functionally-fitted Runge-Kutta methods, {\it BIT Numer. Math.}, 46 (4) (2006), 861--874.
 
 
  
\bibitem{nguyen3}
N. S. Hoang, R. B. Sidje, On the stability of functionally-fitted Runge-Kutta methods, {\it BIT Numer. Math.}, 48 (1) (2008), 61--77.   
 
 \bibitem{nguyenfeptrk}
 N. S. Hoang, R. B. Sidje, Functionally-fitted pseudo two-step Runge-Kutta methods, {\it Appl. Numer. Math.}, 59 (1) (2009), 39--55.
 



\bibitem{Ozawa2} 
K. Ozawa, Functional fitting Runge-Kutta-Nystr\"{o}m method with variable coefficients, {\it Japan J. Indust. App. Math.}, 19 (2002), 55--85.



\bibitem{simos}
G. Psihoyios, T.E. Simos, Trigonometrically fitted predictor-corrector methods for IVPs with oscillating solutions, {\it J. Comp. Appl. Math.}, 158 (1) (2003), 135--144.




\bibitem{cong91}
P. J. van der Houwen, B. P. Sommeijer, N. H. Cong, Stability of collocation-based Runge-Kutta-Nystr\"{o}m methods, {\it BIT Numer. Math.}, 31 (3) (1991), 469--481.

\bibitem{VR} J. Vigo-Aguiar, H. Ramos, Variable step-size implementation of Multistep Methods for
$y''= f(x,y,y')$, \emph{ J. Comput. Appl. Math.}, 192 (2006), 114--131.

\bibitem{Tre} L. N. Trefethen, \emph{ Spectral methods in MATLAB, Software, Environments, and Tools}, Vol. 10,
SIAM, Philadelphia, PA, 2000.


 
\end{thebibliography}
\end{document}